\documentclass[journal]{IEEEtran}

\ifCLASSINFOpdf

\else

\fi

\usepackage{graphicx}          

\usepackage[dvips]{epsfig}    
\usepackage{latexsym,amssymb}
\usepackage{amsmath, amsbsy, amsthm}
\usepackage{amsopn, amstext}
\usepackage{colortbl}
\usepackage{cite}
\newtheorem{remark}{Remark}
\newtheorem{theorem}{Theorem}

\newtheorem{lemma}{Lemma}

\newtheorem{problem}{Problem}
\newtheorem{assumption}{Assumption}
\newtheorem{corollary}{Corollary}
\allowdisplaybreaks
\title{\LARGE \bf
Optimal Local and Remote Controls of Multiple Systems with Multiplicative Noises and Unreliable Uplink Channels
}

\author{ Qingyuan Qi, Lihua Xie*, and Huanshui Zhang
\thanks{This work was supported by Agency for Science, Technology and Research of Singapore under grant A1788a0023, National Natural Science Foundation of China under grants 61903210, 61633014, 61873179, Natural Science Foundation of Shandong Province under grant ZR2019BF002, China Postdoctoral Science Foundation under grant 2019M652324, and Qingdao Postdoctoral Application Research Project. (Corresponding author: Lihua Xie.)

Q. Qi (qiqy123@163.com) is with Institute of Complexity Science, College of Automation, Qingdao University, Qingdao, China 266071, and also with  School of Electrical and Electronic Engineering, Nanyang Technological University, Singapore 639798. L. Xie (ELHXIE@ntu.edu.sg) is with School of Electrical and Electronic Engineering, Nanyang Technological University, Singapore 639798. H. Zhang (hszhang@sdu.edu.cn) is with School of Control Science and Engineering, Shandong University, Jinan, China, 250061.
}
}

\begin{document}

\maketitle
\thispagestyle{empty}
\pagestyle{empty}

\begin{abstract}
In this paper, the optimal local and remote linear quadratic (LQ) control problem is studied for {a networked control system (NCS) which consists of multiple subsystems and each of which is described by a general multiplicative noise stochastic system with one local controller and one remote controller}. Due to the unreliable uplink channels, the remote controller {can only access} unreliable state information of all subsystems, while the downlink channels from {the remote controller to the local controllers} are perfect. The difficulties of the LQ control problem {for such a system} arise from the different information structures of {the local controllers and the remote controller}. By developing the Pontyagin maximum principle, the necessary and sufficient solvability conditions are derived, which are based on the solution to a group of forward and backward difference equations (G-FBSDEs). Furthermore, by proposing a new method to decouple the G-FBSDEs and introducing new coupled Riccati equations (CREs), the optimal control strategies are derived {where we verify} that the separation principle holds for {the} multiplicative noise {NCSs with packet dropouts}. This paper can be seen as an important {contribution to the} optimal control problem with asymmetric information structures.
\end{abstract}
\begin{IEEEkeywords}
Multiplicative noise system, multiple subsystems, optimal local and remote controls, Pontryagin maximum principle.
\end{IEEEkeywords}


\section{Introduction}
As is well known, NCSs are systems in which actuators, sensors, and controllers exchange information through {a shared bandwidth limited} digital communication network. The research on NCSs has attracted {significant interest} in recent years, due to the {advantages of NCSs such as low cost and simple installation}, see \cite{hnx2007,iyb2006,wyb2002,ssfps2007,zbp2001,qz2017,qz2018,xxq2012} and the cited references therein. {On the other hand, unreliable wireless communication channels and limited bandwidth make NCSs less reliable \cite{wyb2002,ssfps2007,zbp2001}}, which may cause performance loss and destabilize the NCSs. Thus, it is necessary to investigate control problems for NCSs with unreliable communication channels.

In this paper, {we investigate multiplicative noise NCSs (MN-NCSs) with local and remote controls and unreliable uplink channels.} As shown in Figure 1, the NCS is composed of $L$ subsystems, and each subsystem is regulated by one local controller and one remote controller. {Due to the limited communication capability of each subsystem, the uplink channel from a local controller to the remote controller is affected by packet dropouts, while the downlink channel from the remote controller to each local controller is perfect. The LQ optimal control problem is considered in this paper. Our aim is to design $L$ local controllers and the remote controller such that a given cost function is minimized. Due to the uncertainty of the uplink channels, the information sets available to the local controllers and the remote controller are different. Hence, we are concerned with the optimal control problem with an asymmetric information structure which pose major challenges.}
 \begin{figure}[htbp]\label{fig1:1}
  \centering
  \includegraphics[width=0.45\textwidth]{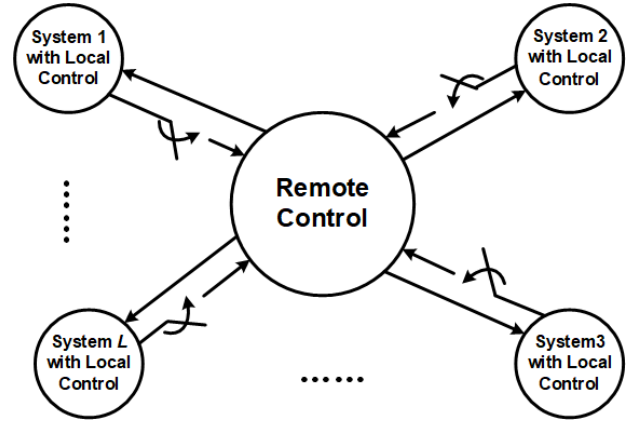}
  \caption{System model: Multiple subsystem NCSs with unreliable uplink channels from the local controllers to the remote controller, and perfect downlink channels from the remote controller to the local controllers.}
\end{figure}

The pioneering study of asymmetric information control can be traced back to the 1960s, and it is well-known that finding optimal controls for an asymmetric-information control problem is difficult, see \cite{w1968b,b2008,rrj2016,lm2011}. For example, \cite{w1968b} gives the celebrated Witsenhausen's Counterexample {for which} the associated explicit optimal control with asymmetric information structures still remains to be solved.

The optimal local and remote control problems were first studied in \cite{nmt2013, aon2018}. By using {the common information approach}, the optimal local and remote controls are derived. {Following \cite{aon2018}, an elementary proof of the common information approach} was given in \cite{am2019}. Furthermore, \cite{lx2018} studies the optimal local and remote control problem with only one subsystem by using maximum principle approach. It should be pointed out that previous works \cite{aon2018,oan2016,lx2018, oan2018b, am2019, nmt2013} {focused on NCSs with additive noise}. Different from {the} previous works, this paper investigates the local control and remote control problem for networked systems with multiplicative noise and multiple subsystems. The motivations of this paper are: On one hand, multiplicative noise systems exist in many applications. The existence of multiplicative noise results in the non-Gaussian property of the NCSs, see \cite{j1976,b1976, zlxf2015,rz2000,w1968, qzw2019}. On the other hand, the unreliable uplink {channels result in that the state estimation errors are} involved with the control inputs, which may {lead to the failure of the separation principle}. In other words, it remains difficult to design the optimal output feedback control for MN-NCSs with unreliable uplink channels, see \cite{j1976,aon2018,lx2018, oan2018b}. Furthermore, the optimal local and remote control problem for multiple subsystems can be regarded as a special case of {optimal control for multi-agent systems}, {for which the} optimal decentralized/distributed control design remains challenging, see \cite{yx2010,ml2014}.

Based on the above discussions, we can conclude {that the study of the} local and remote control problem for MN-NCSs with multiple subsystems {has} the following difficulties and challenges: 1) Due to the possible failure of {the} separation principle for MN-NCSs with unreliable uplink channels, the {derivation of the} optimal ``output feedback" controllers remains challenging. 2) {When the Pontryagin maximum principle is adopted}, how to decouple a group of forward and backward difference equations (G-FBSDEs) is difficult and unsolved.

In this paper, the optimal local and remote control problem for MN-NCSs with multiple subsystems is solved. Firstly, by developing the Pontryagin maximum principle, we show that the optimal control problem under consideration is uniquely solved if and only if a group of FBSDEs (G-FBSDEs) is uniquely solved; Consequently, a method is proposed to decouple the G-FBSDEs, it is shown that the solution to the original G-FBSDEs can be given by decoupling new G-FBSDEs {and} introducing new information filtrations. Furthermore, the optimal local and remote controllers are derived based on the solution to coupled CREs, which are asymmetric. {As special cases, the additive noise NCSs case, the single subsystem case, and the indefinite weighting matrices case are also investigated}.

As far as we know, the obtained results are new and innovative in the following aspects: 1) Compared with the common information approach adopted in previous works, the structure of the optimal local controllers and remote controller is not assumed in advance, see Lemma 3 in \cite{am2019}; 2) In this paper, the multiple subsystems case is solved by using {the} Pontragin Maximum principle, while previous works focused on {the} single subsystem case \cite{lx2018}; 3) The existence of multiplicative noise results in that the state estimate error and the error covariance are involved with the controls, {which may cause the separation principle to fail}. To overcome this, new {asymmetric} CREs are introduced; It is verified that {the} separation principle holds for the considered optimization problem, i.e., the optimal local controllers and optimal remote controller can be designed, and the control gain matrices and the estimation gain matrix can be calculated separately;  4) It is noted that the obtained results {include} the results in \cite{lx2018,am2019,aon2018} as special cases.

The rest of the paper is organized as {follows}. Section II investigates the existence of the optimal control strategies. The main results are presented in Section III, {where} the optimal local controllers and optimal remote controller are derived by decoupling the G-FBSDEs. Consequently, some discussions are given in Section IV. {In Section V, numerical examples are presented to illustrate the main results. We conclude this paper in Section VI. Finally, the proof of the main results is given in the Appendix.}

The {following} notations will be used throughout this paper: $\mathbb{R}^n$ denotes the $n$-dimensional Euclidean space, and $A^T$ means the transpose of matrix $A$; {The subscript of $A^i, i=1,\cdots, L$ means the $i$-th subsystem, and subscript of $A^{\textbf{k}}$ denotes $A$ to the power of $k$.} Symmetric matrix $M>0$ ($\geq 0$) {means that $M$ is positive definite (positive semi-definite)}; $\mathbb{R}^{n}$ {is the} $n$-dimensional Euclidean space; $I_n$ denotes the identity matrix {of dimension $n$}; $E[\cdot]$ means the mathematical expectation and $E[\cdot|Y]$ signifies the conditional expectation with respect to $Y$. $\mathcal{N}(\mu, \Sigma)$ {is} the normal distribution with mean $\mu$ and covariance $\Sigma$, and $Pr(A)$ denotes the probability {of the occurrence of event $A$}. $\mathcal{F}(X)$ means the filtration generated by random variable/vector $X$, $Tr(\cdot)$ means the trace of a matrix, $vec(x^1, x^2,x^3,\cdots)$ means the vector $[(x^1)^T, (x^2)^T, (x^3)^T, \cdots]^T$, and $\sigma(X)$ denotes the $\sigma$-algebra generated by random vector $X$.

\section{Existence of Optimal Control Strategy}

\subsection{Problem Formulation}
The system dynamics of the $i$-th subsystem ($i=1,\cdots,L$) is given by
\begin{align}\label{sm-1}
  x_{k+1}^i=&
  [A^i+w^i_k\bar{A}^i]x^i_k+[B^{i}+w^i_k\bar{B}^{i}]u^i_k\notag\\
  &+[B^{i0}+w^i_k\bar{B}^{i0}]u^0_k+v^i_k,
\end{align}
where $x^i_k\in \mathbb{R}^{n_i}$ is the system state of the $i$-th subsystem at time $k$, $A^i, \bar{A}^i\in\mathbb{R}^{n_i\times n_i}, B^{i},\bar{B}^i\in \mathbb{R}^{n_i\times m_i}, B^{i0},\bar{B}^{i0}\in \mathbb{R}^{n_i\times m_{0}}$ are the given system matrices, both $w^i_k\in \mathbb{R}$ and $v^i_k\in \mathbb{R}^{n}$ are Gaussian {white noises satisfying $w_k^i\sim\mathcal{N}(0,\Sigma_{w^i}), v^i_k\sim\mathcal{N}(0,\Sigma_{v^i})$}. The initial state $x^i_0\sim\mathcal{N}(\mu^i,\Sigma_{x_0^i})$. $u^i_k\in\mathbb{R}^{m_i}$ is {the} local controller of {the} $i$-th subsystem, and $u^0_k\in\mathbb{R}^{m_{0}}$ is the remote controller. {Since the uplink channels are unreliable, let binary random variables $\gamma_k^i, i=1, 2, ...$ with probability distribution $Pr(\gamma_k^i=1)=p^i$ denote if a packet is successfully transmitted, i.e., $\gamma_k^i=1$ means {that} the packet is successfully transmitted from the $i$-th local controller to the remote controller, and {fails} otherwise. $p^i, i=1,\cdots,L$ is called the packet dropout rate. Moreover, we assume $x_0, \{\gamma_k^i\}_{k=0}^{N}, \{w_k^i\}_{k=0}^{N}, \{v_k^i\}_{k=0}^{N} $ are independent of each other for all $i=1,\cdots,L$.}

For the sake of discussion, the following notions are denoted:
{\small\begin{align}\label{nota1}
 X_k&= vec(x_k^1,\cdots, x_k^L), U_k=vec(u^0_k, u_k^1, \cdots, u_k^L),\notag\\
 V_k&= vec(v_k^1, \cdots, v_k^L),
\mathbb{N}_L=\sum_{i=1}^{L}n_i, \mathbb{M}_L=\sum_{i=0}^{L}m_i,\notag\\
A&=\left[\hspace{-2mm}
  \begin{array}{cccc}
    A^1&\cdots&0\\
    \vdots&\ddots&\vdots\\
    0&\cdots&A^L
  \end{array}
\hspace{-2mm}\right],B=\left[\hspace{-2mm}
  \begin{array}{cccc}
    B^{10}&B^1&\cdots&0\\
    \vdots&\vdots&\ddots&\vdots\\
    B^{L0}&0&\cdots&B^L
  \end{array}
\hspace{-2mm}\right],\notag\\
\bar{\mathbf{A}}^i&=\left[\hspace{-2mm}
  \begin{array}{ccccc}
  \text{\Huge{0}}&&\cdots &&\text{\Huge{0}}\\
  &\ddots&&&\\
    \vdots&&\bar{A}^i &&\vdots\\
    &&&\ddots&\\
   \text{\Huge{0}}&&\cdots &&\text{\Huge{0}}
  \end{array}
\hspace{-2mm}\right]_{\mathbb{N}_L\times \mathbb{N}_L},\notag\\
\bar{\mathbf{B}}^i&=\left[\hspace{-2mm}
  \begin{array}{cccccc}
  \bar{B}^{10}&\text{\Huge{0}}&&\cdots &&\text{\Huge{0}}\\
  &&\ddots&&&\\
   \vdots& \vdots&&\bar{B}^i &&\vdots\\
    &&&&\ddots&\\
   \bar{B}^{L0}&\text{\Huge{0}}&&\cdots &&\text{\Huge{0}}
  \end{array}
\hspace{-2mm}\right]_{\mathbb{N}_L\times \mathbb{M}_L},\notag\\
\Gamma_k&=\left[\hspace{-2mm}
  \begin{array}{cccc}
  \gamma^1_kI_{n_1}&\cdots&0\\
    \vdots&\ddots&\vdots\\
    0&\cdots&\gamma^L_kI_{n_1}
  \end{array}
\hspace{-2mm}\right], p=\left[\hspace{-2mm}
  \begin{array}{cccc}
  p^1I_{n_1}&\cdots&0\\
    \vdots&\ddots&\vdots\\
    0&\cdots&p^LI_{n_L}
  \end{array}
\hspace{-2mm}\right].
\end{align}}
Using the notations in \eqref{nota1}, we can rewrite system \eqref{sm-1} as
\begin{align}\label{asys}
 X_{k+1}&=AX_k+B U_k+\sum_{i=1}^{L} w_k^i(\bar{\mathbf{A}}^iX_k+\bar{\mathbf{B}}^iU_k)+V_k
\end{align}
with initial $X_0=vec(x_0^1, \cdots, x_0^L)$.

Associated with system \eqref{sm-1}, the following cost function is introduced
\begin{align}\label{pi}
  J_N & \hspace{-1mm}=\hspace{-1mm}E\left[\sum_{k=0}^{N}(X^T_kQX_k
  \hspace{-1mm}+\hspace{-1mm}U^T_kRU_k)\hspace{-1mm}+\hspace{-1mm}X_{N+1}^TP_{N+1}X_{N+1}\right],
\end{align}
where $Q,R, P_{N+1}$ are symmetric weighting matrices of appropriate dimensions with
\begin{align}\label{nota2}
  Q&=\left[\hspace{-2mm}
  \begin{array}{cccc}
    Q^{11}&\cdots&Q^{1L}\\
    \vdots&\ddots&\vdots\\
    Q^{L1}&\cdots&Q^{LL}
  \end{array}
\hspace{-2mm}\right], R=\left[\hspace{-2mm}
  \begin{array}{cccc}
    R^{00}&\cdots&R^{0L}\\
    \vdots&\ddots&\vdots\\
   R^{L0}&\cdots&R^{LL}
  \end{array}
\hspace{-2mm}\right],\notag\\
P_{N+1}&=\left[\hspace{-2mm}
  \begin{array}{cccc}
    P_{N+1}^{11}&\cdots&P_{N+1}^{1L}\\
    \vdots&\ddots&\vdots\\
    P_{N+1}^{L1}&\cdots&P_{N+1}^{LL}
  \end{array}
\hspace{-2mm}\right],
\end{align}
and block matrices $Q^{ij}, P_{N+1}^{ij}\in\mathbb{R}^{n_i\times n_j}, i,j=1,\cdots,L$, $R^{ij}\in\mathbb{R}^{m_i\times m_j}, i,j=0,\cdots,L$.

Corresponding with the system model described in Figure 1, the feasibility of controllers $u_k^i, i=0,\cdots,L$ is given in the following assumption.
\begin{assumption}\label{ass1}
The remote controller $u^0_k$ is $\mathcal{F}^0_k$-measurable, and the local controller  $u^i_k$ is measurable with respect to $\mathcal{F}^i_k$, $i=1,\cdots,L$, where
\begin{align}
  \mathcal{F}^0_k&=\sigma\Big\{\{\Gamma_m\}_{m=0}^{k},
  \{\Gamma_mX_m\}_{m=0}^{k},
  \{u^0_m\}_{m=0}^{k-1}\Big\},\label{mwrt}\\
  \mathcal{F}^i_k&=\sigma\Big\{\{\Gamma_m\}_{m=0}^{k}, \{\Gamma_mX_m\}_{m=0}^{k},
  \{u^0_m\}_{m=0}^{k-1},\notag\\
  &~~~~~~~~~~~~~~~~~~~~~~~~~~~\text{and}~\{u^i_m\}_{m=0}^{k-1},\{x^i_m\}_{m=0}^{k}\Big\}.\label{mwrt2}
\end{align}
\end{assumption}

 It can be easily judged from \eqref{mwrt}-\eqref{mwrt2} that $\mathcal{F}^0_k\subset \mathcal{F}^i_k, i=1,\cdots,L$.

Next, we will introduce the LQ control problem to be solved in this paper.
\begin{problem}\label{prob1}
 For system \eqref{sm-1}-\eqref{asys}, find $\mathcal{F}^i_k$-measurable control $u_k^i$ to minimize cost function \eqref{pi}, where $k=0,\cdots,N, i=0,\cdots,L$.
\end{problem}

Throughout this paper, the assumption on the weighting matrices of \eqref{pi} is given as follows.
\begin{assumption}\label{ass2}
  $Q\geq 0, R>0$, and $P_{N+1}\geq 0$.
\end{assumption}

\begin{remark}
It is stressed that Problem \ref{prob1} has not been solved {in the existing literatures}. {The} previous works \cite{aon2018,am2019,lx2018,oan2016,nmt2013} mainly focused on additive noise {systems, and their} multiplicative noise counterpart remains less investigated. {The existence of unreliable uplink channels for multiplicative noise systems may result in the failure of  ``separation principle", making the design of optimal control {in Problem \ref{prob1}} difficult. Furthermore, the} maximum principle was adopted to solve only the single subsystem case in \cite{lx2018}, while the multiple subsystems case hasn't been solved in the framework of maximum principle. {The main challenge for the multiple subsystems case is that the solution for the G-FBSDEs is difficult.}
\end{remark}


\subsection{Necessary and Sufficient Solvability Conditions}

In this section, we will derive the necessary and sufficient solvability conditions for Problem \ref{prob1}.

\begin{lemma}\label{lem1}
Given system \eqref{sm-1}-\eqref{asys} and cost function \eqref{pi}, for $i=l,\cdots,L, k=0,\cdots,N$, let $u^{i,\varepsilon}_k=u^i_k+\varepsilon \delta u^i_k$, {where $\varepsilon\in \mathbb{R}^1$ and $\delta u_k^i$ is $\mathcal{F}_k^i$-measurable satisfying $\sum_{k=0}^{N}E[ (\delta u^i_k)^T\delta u^i_k]<+\infty$}. Denote $x^{i,\varepsilon}_k$ and $ J_N^\varepsilon$ the corresponding state and cost function {associated with} $u^{i,\varepsilon}_k$, $U^{\varepsilon}_k=vec(u^{0,\varepsilon}_k~\cdots~u^{L,\varepsilon}_k)$, $X^{\varepsilon}_k=vec(x^{0,\varepsilon}_k~\cdots~x^{L,\varepsilon}_k)$, and $\delta U_k=vec(\delta u_k^1, \cdots, \delta u_k^L)$. Then we have
\begin{align}\label{diff}
J_N^\varepsilon-J_N&=\varepsilon^\textbf{2}\Big\{\sum_{k=0}^{N}E\{Y_k^TQY_k+\delta U_kR\delta U_k\}\notag\\
  &~~~~~~~~+E[Y^T_{N+1}P_{N+1}Y_{N+1}]\Big\}\\
  &+2\varepsilon \sum_{k=0}^{N}E\{[\Theta^T_k(B+\sum_{i=1}^{L}w_k^i\bar{\mathbf{B}}^i)+RU_k]^T\delta U_k\},\notag
\end{align}
where {$Y_k$ satisfies the iteration}
\begin{align}\label{yk}
  Y_{k+1}&=[A+\sum_{i=1}^{L}w_k^i\bar{\mathbf{A}}^i]Y_k+[B+\sum_{i=1}^{L}w_k^i\bar{\mathbf{B}}^i]\delta U_k
\end{align}
with initial condition $Y_0=0$, {and} $\Theta_k=vec(\Theta_k^1,\cdots,\Theta_k^L)$ ($k=0,\cdots,N$) satisfies the iteration
\begin{align}\label{coss}
    \Theta_{k-1}&=E\left[(A+\sum_{i=1}^{L}w_k^i\bar{\mathbf{A}}^i)^T\Theta_k+QX_k\Bigg|\mathcal{G}_k\right],
  \end{align}
{with terminal condition $\Theta_N=P_{N+1}X_{N+1}$ and the information filtration $\mathcal{G}_k$ given by}
 \begin{align}\label{maxf}
    \mathcal{G}_k&\hspace{-1mm}=\hspace{-1mm}\sigma\Big\{\{\Gamma_m\}_{m=0}^{k},
  \{U_m\}_{m=0}^{k-1}, \{X_m\}_{m=0}^{k}\Big\}.
  \end{align}
\end{lemma}
\begin{proof}
  By setting $Y_k=\frac{X^\varepsilon_k-X_k}{\varepsilon}$, we know that $Y_0=0$ and \eqref{yk} holds.

  Subsequently, it can be calculated {that}
  \begin{align*}
     & J_N^\varepsilon-J_N\notag\\
     &=\sum_{k=0}^{N}E\Big[[X_k+\varepsilon Y_k]^TQ[X_k+\varepsilon Y_k]\notag\\
     &+[U_k+\varepsilon \delta U_k]^TR[U_k+\varepsilon \delta U_k]\Big]\notag\\
     &+E[X_{N+1}+\varepsilon Y_{N+1}]^TP_{N+1}[X_{N+1}+\varepsilon Y_{N+1}]\notag\\
     &-EX^T_{N+1}P_{N+1}X_{N+1}
     -\sum_{k=0}^{N}E\Big[X^T_kQX_k+U^T_kRU_k\Big]\notag\\
     &=2\varepsilon E[\sum_{k=0}^{N}[X^T_kQY_k+\delta U^T_kRU_k]
     +Y^T_{N+1}P_{N+1}X_{N+1} ]\notag\\
     &+\varepsilon^\textbf{2}E[\sum_{k=0}^{N}[Y^T_kQY_k+\delta U^T_kR \delta U_k]+Y^T_{N+1}P_{N+1}Y_{N+1}].
  \end{align*}

  {Using \eqref{yk}-\eqref{coss}, and noting $Y_k$ is $\mathcal{G}_k$-measurable}, we have
  \begin{align*}
   &E[\sum_{k=0}^{N}[X^T_kQY_k+\delta U^T_kRU_k]
   +Y^T_{N+1}P_{N+1}X_{N+1} ]\notag\\
     &=E\Big[\sum_{k=0}^{N}\{\Theta_{k-1}-E[(A+\sum_{i=1}^{L}w_k^i\bar{\mathbf{A}}^i)^T\Theta_k|\mathcal{G}_k]\}^TY_k\notag\\
     &+\delta U^T_kRU_k+\Theta^T_NY_{N+1} \Big]\notag\\
     &=E\Big[\sum_{k=0}^{N}[\Theta^T_k(B+\sum_{i=1}^{L}w_k^i\bar{\mathbf{B}}^i)+RU_k]^T\delta U_k\Big],
  \end{align*}
  which ends the proof.
\end{proof}

Accordingly, we have the following results.
\begin{theorem}\label{th1}
Under Assumptions \ref{ass1} and \ref{ass2}, Problem \ref{prob1} can be uniquely solved if and only if the equilibrium condition
  \begin{align}\label{coseq}
    0&=RU_k+E\left[(B+\sum_{i=1}^{L}w_k^i\bar{\mathbf{B}}^i)^T\Theta_k\Bigg|\mathcal{G}_k \right]
  \end{align}
  {can be uniquely solved,} where the costate $\Theta_k$ satisfies \eqref{coss}.
\end{theorem}
\begin{proof}
`Necessity': Suppose Problem \ref{prob1} is uniquely {solvable and} $U_k=vec(u^{1}_k, \cdots, u_k^L)$ are optimal controls for $k=0,\cdots,N$.

Using the symbols in Lemma \ref{lem1} and from \eqref{diff} we know that, for arbitrary $\delta U_k$ and $\varepsilon\in \mathbb{R}$, there holds
\begin{align}\label{diff2}
J_N^\varepsilon-J_N&=\varepsilon^\textbf{2}\delta J_N\notag\\
  &+2\varepsilon \sum_{k=0}^{N}E\Big[[\Theta^T_k(B+\sum_{i=1}^{L}w_k^i\bar{\mathbf{B}}^i)+RU_k]^T\delta U_k\Big]\notag\\
  &\geq 0,
\end{align}
where
\begin{align}\label{deljn}
  \delta J_N&=\sum_{k=0}^{N}E\Big[Y_k^TQY_k+\delta U_kR\delta U_k\Big]\notag\\
  &~~~~~~~~+E[Y^T_{N+1}P_{N+1}Y_{N+1}].
\end{align}

{Observe from Assumption \ref{ass2} that} $\delta J_N\geq 0$. Next we will show \eqref{coseq} holds.

Suppose, {by contradiction, that \eqref{coseq} is not satisfied. Let}
  \begin{align} \label{u11}
    &RU_k+E\left[[B+\sum_{i=1}^{L}w_k^i\bar{\mathbf{B}}^i]^T\Theta_k\Bigg|\mathcal{G}_k\right]
    =\tau_k \neq 0.
  \end{align}
  In this case, if we choose $\delta U_k=\tau_k$, then from \eqref{diff2} we have
  \begin{align*}
  J_N^\varepsilon-J_N= 2\varepsilon\sum_{k=0}^{N}
   \tau_k^T\tau_k +\varepsilon^2\delta J_N.
  \end{align*}
{Note that} we can always find some $\varepsilon<0$ such that {$J_N^\varepsilon-J_N<0$}, which contradicts with \eqref{diff2}. Thus, $\tau_k=0$. This ends the necessity proof.

`Sufficiency': If \eqref{coseq} is uniquely solvable, we shall show that Problem \ref{prob1} is uniquely {solvable} under Assumptions \ref{ass1}-\ref{ass2}.

Actually, from \eqref{diff} we know that for any $\varepsilon\in\mathbb{R}$, $J_N^\varepsilon-J_N=\varepsilon^2\delta J_N\geq 0$, which means that Problem \ref{prob1} is uniquely solvable.
\end{proof}

It is noted that system dynamics \eqref{sm-1} and \eqref{asys} are forward, and the costate equation \eqref{coss} is backward, then \eqref{sm-1}, \eqref{asys}, \eqref{coss} and \eqref{coseq} constitute the G-FBSDEs. For the convenience of discussion, we denote the following G-FBSDEs composed of \eqref{sm-1}, \eqref{asys}, \eqref{coss} and \eqref{coseq}:
{\small\begin{equation}\label{gfbs}
\left\{ \begin{array}{ll}
  x_{k+1}^i&=
  [A^i+w^i_k\bar{A}^i]x^i_k+[B^{i}+w^i_k\bar{B}^{i}]u^i_k\\
  &+[B^{i0}+w^i_k\bar{B}^{i0}]u^0_k+v^i_k,~x_0^i, i=1,\cdots,L,\\
  X_{k+1}&=(A+\sum_{i=1}^{L}w_k^i\bar{\mathbf{A}}^i)X_k\\
  &+(B+\sum_{i=1}^{L}w_k^i\bar{\mathbf{B}}^i)U_k+V_k,\\
  \Theta_{k-1}&=E\left[(A+\sum_{i=1}^{L}w_k^i\bar{\mathbf{A}}^i)^T\Theta_k+QX_k|\mathcal{G}_k\right],\\
  \Theta_N&=P_{N+1}X_{N+1},\\
  0&=RU_k+E\left[(B+\sum_{i=1}^{L}w_k^i\bar{\mathbf{B}}^i)^T\Theta_k\Big|\mathcal{G}_k \right].
\end{array} \right.
\end{equation}}

\begin{remark}
  The necessary and sufficient solvability conditions of Problem \ref{prob1} {given in Theorem 1} are presented for the first time, which are based on the solution to G-FBSDEs \eqref{gfbs}. Consequently, to derive the optimal control strategies $u_k^i, i=0\cdots,L, k=0,\cdots,N$, we will {find a} method of decoupling G-FBSDEs \eqref{gfbs}.
\end{remark}

Consequently, we will introduce some preliminary results.
\begin{lemma}\label{lem3}
{Denote} $\hat{u}_k^i=E[u_k^i|\mathcal{F}_k^0]$, then the following relationship holds:

  \begin{align}\label{rela1}
E[u_k^j|\mathcal{G}_k^i]=\left\{ \begin{array}{ll}
\hat{u}_k^i,&j=i,\\
u_k^j, &j\neq i,
\end{array} \right.
  \end{align}
  where the information filtration $\mathcal{G}_k^i$ is given by
  \begin{align}\label{hki}
   \mathcal{G}_k^i&=\sigma\Big\{\{\Gamma_m\}_{m=0}^{k},
  \{\Gamma_mX_m\}_{m=0}^{k},\{U_m\}_{m=0}^{k-1},\\
  &~~~~~~~\text{and}~\{x_m^j\}_{m=0}^{k},j=1,\cdots,i-1,i+1,\cdots,L\Big\}.\notag
  \end{align}
\end{lemma}
\begin{proof}
  Due to the independence of $u_k^i$ and $\{x_m^j\}_{m=0}^{k}$, $\{u_m^j\}_{m=0}^{k}, j\neq i$, \eqref{rela1} can be obtained by using the properties of conditional expectation.
\end{proof}


By using Lemma \ref{lem3}, the following result can be derived.
\begin{lemma}\label{lem4}
Under Assumptions \ref{ass1} and \ref{ass2}, the equilibrium condition \eqref{coseq} can be rewritten as:
\begin{align}
  0&=R\hat{U}_k+E[(B+\sum_{i=1}^{L}w_k^i\bar{\mathbf{B}}^i)^T\Theta_k|\mathcal{F}^0_k],\label{adeqs}\\
  0 &=R\tilde{U}_{k}+E[(B+\sum_{i=1}^{L}w_k^i\bar{\mathbf{B}}^i)^T\Theta_{k}|\mathcal{G}_{k}]\notag\\
    &~~~~~~~~~~~~~-E[(B+\sum_{i=1}^{L}w_k^i\bar{\mathbf{B}}^i)^T\Theta_{k}|\mathcal{F}^0_{k}], \label{2adeqs}\\
  0&=R^{ii}\tilde{u}_k^i+E[(B^i+w_k^i\bar{B}^i)^T\Theta_k^i|\mathcal{G}_k]\notag\\
  &~~~~~~~~~~~~~-E[(B^i+w_k^i\bar{B}^i)^T\Theta_k^i|\mathcal{G}_k^i],\label{adeqs2}
\end{align}
where $ \hat{U}_k=vec(u^0_k,\hat{u}^1_k,\cdots, \hat{u}^L_k)$, $\tilde{U}_k=U_k-\hat{U}_k$ and $ \Theta_k=vec(\Theta^0_k,\Theta^1_k,\cdots, \Theta^L_k)$ satisfies \eqref{coss}.
\end{lemma}
\begin{proof}
The results can be easily derived from Lemma \ref{lem3}.
\end{proof}

In view of Lemma \ref{lem4}, the G-FBSDEs \eqref{gfbs} can be equivalently presented as
{\small\begin{equation}\label{gfbs2}
\left\{ \begin{array}{ll}
  x_{k+1}^i&=
  [A^i+w^i_k\bar{A}^i]x^i_k+[B^{i}+w^i_k\bar{B}^{i}]u^i_k\\
  &+[B^{i0}+w^i_k\bar{B}^{i0}]u^0_k+v^i_k,\\
  X_{k+1}&=(A+\sum_{i=1}^{L}w_k^i\bar{\mathbf{A}}^i)X_k\\
  &+(B+\sum_{i=1}^{L}w_k^i\bar{\mathbf{B}}^i)U_k+V_k,\\
  \Theta_{k-1}&=E\left[(A+\sum_{i=1}^{L}w_k^i\bar{\mathbf{A}}^i)^T\Theta_k+QX_k|\mathcal{G}_k\right],\\
  \Theta_N&=P_{N+1}X_{N+1},\\
  0&=R\hat{U}_k+E[(B+\sum_{i=1}^{L}w_k^i\bar{\mathbf{B}}^i)^T\Theta_k|\mathcal{F}^0_k],\\
  0 &=R\tilde{U}_{k}+E[(B+\sum_{i=1}^{L}w_k^i\bar{\mathbf{B}}^i)^T\Theta_{k}|\mathcal{G}_{k}]\\
    &-E[(B+\sum_{i=1}^{L}w_k^i\bar{\mathbf{B}}^i)^T\Theta_{k}|\mathcal{F}^0_{k}],\\
  0&=R^{ii}\tilde{u}_k^i+E[(B^i+w_k^i\bar{B}^i)^T\Theta_k^i|\mathcal{G}_k]\\
  &-E[(B^i+w_k^i\bar{B}^i)^T\Theta_k^i|\mathcal{G}_k^i].
\end{array} \right.
\end{equation}}

In the following lemma, we will introduce the preliminary results on the optimal estimation and the associated state estimation error.
\begin{lemma}\label{lemma2}
The optimal estimation $\hat{x}^i_k\triangleq E[x^i_k|\mathcal{F}^0_k], i=1,\cdots, L$ and $\hat{X}_k\triangleq E[X_k|\mathcal{F}^0_k]$  can be calculated by
\begin{align}
    \hat{x}^i_{k+1} &=
    (1-\gamma^i_{k+1})(A^i\hat{x}^i_k+B^i\hat{u}^i_k+B^{i0}u^0_k)\notag\\
    &+\gamma^i_{k+1}x^i_{k+1},\label{oe1}\\
\hat{X}_{k+1}&=(I_{\mathbb{N}_L}-\Gamma_{k+1})(A\hat{X}_k+B\hat{U}_k)+\Gamma_{k+1} X_{k+1},\label{oe2}
  \end{align}
 { with initial conditions $\hat{x}_{0}^i=\gamma^i_{0}\mu^i+(1-\gamma_{0}^i)x_{0}^i$, $\hat{X}_0=(I_{\mathbb{N}_L}-\Gamma_{0})\mu+\Gamma_{k+1} X_{0}$ and $\mu=vec(\mu^1,\cdots,\mu^L)$.}

  In this case, the error covariance $\tilde{x}_k^i=x_k^i-\hat{x}_k^i, i=1,\cdots,L$ and $\tilde{X}_k=X_k-\hat{X}_k$ satisfy
  \begin{align}
   \tilde{x}_{k+1}^i&=(1-\gamma^i_{k+1})\Big[ A^i\tilde{x}^i_{k}+B\tilde{u}^i_{k}\notag\\
   &+w_{k}^i(\bar{A}^ix_{k}^i+\bar{B}^i(\hat{u}^i_{k}+\tilde{u}^i_k))+v_k^i\Big],\label{tild1}\\
   \tilde{X}_{k+1}&=(I_{\mathbb{N}_L}-\Gamma_{k+1})\Big[ A\tilde{X}_{k}+B\tilde{U}_{k}\notag\\
   &+\sum_{i=1}^{L}w_k^i(\bar{\mathbf{A}}^iX_{k}
   +\bar{\mathbf{B}}^i(\hat{U}_{k}+\tilde{U}_k))+V_{k}\Big].\label{tild2}
  \end{align}

\end{lemma}
\begin{proof}
  The detailed proof can be found in \cite{ssfps2007, qz2017}, which is omitted here.
\end{proof}

\begin{remark}
 It can be {observed from \eqref{tild1}-\eqref{tild2} that the controls are involved with the state estimation error, which} is the key difference from the additive noise case (i.e., $w_k^i=0$), see \cite{aon2018,am2019,lx2018,oan2016,nmt2013}. Moreover, since G-FBSDEs \eqref{gfbs} and G-FBSDEs \eqref{gfbs2} {are equivalent}, we will derive the optimal controls by solving \eqref{gfbs2} instead.
\end{remark}

\section{Optimal Controls by Decoupling G-FBSDEs}
In this section, the optimal controls $u_k^i, i=0,\cdots,L$ will be derived via decoupling the G-FBSDEs \eqref{gfbs} (equivalently G-FBSDEs \eqref{gfbs2}).

Firstly, the following CREs are introduced:
{\small\begin{equation}\label{re}
\left\{ \begin{array}{ll}
  P_k&=Q+A^TP_{k+1}A
  +\sum_{i=1}^{L}\Sigma_{w^i}(\bar{\mathbf{A}}^i)^TL_{k+1}
  \bar{\mathbf{A}}^i\\
  &-\Psi_k^T\Lambda_k^{-1}\Psi_k,\\
   H_k&=Q+A^TL_{k+1}A+\sum_{i=1}^{L}\Sigma_{w^i}(\bar{\mathbf{A}}^i)^TL_{k+1}
  \bar{\mathbf{A}}^i\\
  &-\tilde{\Psi}_k^T\tilde{\Lambda}_k^{-1}\tilde{\Psi}_k,\\
    P_k^{i}&=Q^{ii}+(A^i)^TP_{k+1}^{i}A^i
    +\Sigma_{w^i}(\bar{A}^i)^TL_{k+1}^{i}\bar{A}^i\\
    &-(\Omega_k^i)^T(\Pi_k^i)^{-1}\Omega_k^i,\\
   H_k^{i}&=Q^{ii}+(A^i)^TL_{k+1}^{i}A^i
    +\Sigma_{w^i}(\bar{A}^i)^TL_{k+1}^{i}\bar{A}^i\\
    &-(\tilde{\Omega}_k^i)^T(\tilde{\Pi}_k^i)^{-1}\tilde{\Omega}_k^i,
\end{array} \right.
\end{equation}}
with terminal conditions $L_{N+1}=H_{N+1}=P_{N+1}$, $L_{N+1}^{i}=H_{N+1}^{i}=P_{N+1}^{i}$ given in
\eqref{pi}, and the coefficients matrices $\Lambda_k, \Psi_k, \tilde{\Lambda}_k, \tilde{\Psi}_k, \Pi_k^i, \Omega_k^i, \tilde{\Pi}_k^i, \tilde{\Omega}_k^i, L_k, L_k^{i}$ in \eqref{re} {being} given by
{\small\begin{equation}\label{coma}
\left\{ \begin{array}{ll}
 \Lambda_k&=R+B^TP_{k+1}B+\sum_{i=1}^{L}\Sigma_{w^i}(\bar{\mathbf{B}}^i)^TL_{k+1}
  \bar{\mathbf{B}}^i,\\
    \Psi_k&=B^TP_{k+1}A+\sum_{i=1}^{L}\Sigma_{w^i}(\bar{\mathbf{B}}^i)^TL_{k+1}
  \bar{\mathbf{A}}^i,\\
    \tilde{\Lambda}_k&=R+B^TL_{k+1}B+\sum_{i=1}^{L}\Sigma_{w^i}(\bar{\mathbf{B}}^i)^TL_{k+1}
  \bar{\mathbf{B}}^i,\\
    \tilde{\Psi}_k&=B^TL_{k+1}A+\sum_{i=1}^{L}\Sigma_{w^i}(\bar{\mathbf{B}}^i)^TL_{k+1}
  \bar{\mathbf{A}}^i,\\
    \Pi^i_k&=R^{ii}+(B^i)^TP_{k+1}^{i}B^i+\Sigma_{w^i}(\bar{B}^i)^TL_{k+1}^{i}\bar{B}^i,\\
  \Omega_k^i&=(B^i)^TP_{k+1}^{i}A^i+\Sigma_{w^i}(\bar{B}^i)^TL_{k+1}^{i}\bar{A}^i,\\
    \tilde{\Pi}^i_k&=R^{ii}+(B^i)^TL_{k+1}^{i}B^i+\Sigma_{w^i}(\bar{B}^i)^TL_{k+1}^{i}\bar{B}^i,\\
  \tilde{\Omega}_k^i&=(B^i)^TL_{k+1}^{i}A^i+\Sigma_{w^i}(\bar{B}^i)^TL_{k+1}^{i}\bar{A}^i,\\
  L_k&=P_k p+H_k(I_{\mathbb{N}_L}-p), \\
  L_k^{i}&=p^iP_k^{i}+(1-p^i)H_k^{i}.
\end{array} \right.
\end{equation}}

\begin{remark}
For CREs \eqref{re}, the following points should be noted:
\begin{itemize}
  \item The CREs \eqref{re} are well defined (i.e., can be recursively calculated backwardly) if and only if $\Lambda_k, \tilde{\Lambda}_k$, $\Pi_k^i$ and $\tilde{\Pi}_k^i$ are invertible.
  \item Different from traditional symmetric Riccati equation \cite{lvs2015}, $P_k, H_k$ in \eqref{re} are {asymmetric}, while {$P_k^{i}$ and $H_k^{i}$} are symmetric.
\end{itemize}
\end{remark}

\begin{lemma}\label{lem6}
  Under Assumptions \ref{ass1} and \ref{ass2}, $\Pi_k^i, \tilde{\Pi}_k^i$ are positive definite for $k=0,\cdots,N$.
\end{lemma}
\begin{proof}
  To facilitate the discussions, we denote:
  \begin{align}\label{gki}
    g_k^i&=-(\Pi_k^i)^{-1}\Omega_k^i,~~ \tilde{g}_k^i=-(\tilde{\Pi}_k^i)^{-1}\tilde{\Omega}_k^i.
  \end{align}

  With $k=N$, from \eqref{re}-\eqref{gki} we know that
  {\small\begin{equation}\label{euqil}
\left\{ \begin{array}{ll}
 (\Omega_N^i)^T(\Pi_N^i)^{-1}\Omega_N^i&=(g_N^i)^T\Pi_N^ig_N^i,\\
 &=-(g_N^i)^T\Omega_N^i=-(\Omega_N^i)^Tg_N^i,\\
 (\tilde{\Omega}_N^i)^T(\tilde{\Pi}_N^i)^{-1}\tilde{\Omega}_N^i
 &=(\tilde{g}_N^i)^T\tilde{\Pi}_N^i\tilde{g}_N^i,\\
 &=-(\tilde{g}_N^i)^T\tilde{\Omega}_N^i=-(\tilde{\Omega}_N^i)^T\tilde{g}_N^i.
\end{array} \right.
\end{equation}}

Then we can rewrite $P_N^{i}$ in CREs \eqref{re} as
\begin{align}\label{pd1}
  P_N^{i}&=Q^{ii}+(A^i)^TP_{N+1}^{i}A^i
    +\Sigma_{w^i}(\bar{A}^i)^TL_{N+1}^{i}\bar{A}^i\notag\\
    &-(\Omega_N^i)^T(\Pi_N^i)^{-1}\Omega_N^i\notag\\
    &=Q^{ii}+(A^i)^TP_{N+1}^{i}A^i
    +\Sigma_{w^i}(\bar{A}^i)^TL_{N+1}^{i}\bar{A}^i\notag\\
    &+({g}_N^i)^T{\Pi}_N^i{g}_N^i
    +({g}_N^i)^T{\Omega}_N^i+({\Omega}_N^i)^T{g}_N^i\notag\\
    &=Q^{ii}\hspace{-1mm}+\hspace{-1mm}({g}_N^i)^TR^{ii}{g}_N^i
    \hspace{-1mm}+\hspace{-1mm}(A^i+B^i{g}_N^i)^TP_{N+1}^{i}(A^i+B^i{g}_N^i)\notag\\
    &+\Sigma_{w^i}(\bar{A}^i+\bar{B}g_N^i)^TL_{N+1}^{i}(\bar{A}^i+\bar{B}g_N^i).
\end{align}
Similarly, it is not hard to obtain {that}
\begin{align}\label{pd2}
  H_N^{i} & =Q^{ii}\hspace{-1mm}+\hspace{-1mm}(\tilde{g}_N^i)^TR^{ii}\tilde{g}_N^i
    \hspace{-1mm}+\hspace{-1mm}
    (A^i+B^i\tilde{g}_N^i)^TL_{N+1}^{i}(A^i+B^i\tilde{g}_N^i)\notag\\
    &+\Sigma_{w^i}(\bar{A}^i+\bar{B}\tilde{g}_N^i)^T
    L_{N+1}^{i}(\bar{A}^i+\bar{B}\tilde{g}_N^i).
\end{align}

From Assumption \ref{ass2}, {we know that} $P_N^{i}$ and $H_N^{i}$ are both positive semidefinite, then it can be {observed} from \eqref{coma} that $\Pi_N^i>0$ and $\tilde{\Pi}_N^i>0$.

By repeating the above procedures backwardly, we can conclude that $\Pi_k^i$ and $\tilde{\Pi}_k^i$ are positive definite for $k=0,\cdots,N$. This ends the proof.
\end{proof}

{With} the preliminaries introduced in Lemmas \ref{lem1}-\ref{lem6}, {we are in a position to present the solution to Problem \ref{prob1}}.
\begin{theorem}\label{th2}
Under Assumptions \ref{ass1} and \ref{ass2}, Problem \ref{prob1} is uniquely solvable if and only if $\Lambda_k$ and $\tilde{\Lambda}_k$ given by \eqref{coma} are invertible. In this case, the optimal controls $u^i_k, k=0,\cdots,N, i=0,\cdots,L$ of minimizing cost function \eqref{pi} are given by
\begin{equation}\label{uk}
\left\{ \begin{array}{ll}
  u_k^0&=\mathcal{I}^0\hat{U}_k,\\
 u_k^i&=\mathcal{I}^i\hat{U}_k+\tilde{u}_k^i, i=1,\cdots,L,
\end{array} \right.
\end{equation}
where
  {\begin{align}
  \hat{U}_k& =-\Lambda^{-1}_k\Psi_k\hat{X}_k, \label{uk0}\\
   \tilde{u}_k^i&=-(\tilde{\Pi}^i_k)^{-1}\tilde{\Omega}^i_k\tilde{x}_k^i,~~i=1,\cdots,L,\label{uki}\\
\mathcal{I}^0&=[I_{m^0},0_{m^0\times m^1},\cdots,0_{m^0\times m^L}]_{m^0\times \mathbb{M}_{L}},\label{i0}\\
\mathcal{I}^i&=[0_{m^i\times m^0},0_{m^i\times m^1},\cdots,I_{m^i\times m^i},\cdots,0_{m^i\times m^L}]_{m^0\times \mathbb{M}_{L}},\notag
  \end{align}}
 {with $\hat{X}_k, \tilde{x}_k^i$ being calculated from Lemma \ref{lemma2}, and the coefficient matrices $\Lambda_k, \Psi_k, \tilde{\Pi}_k, \tilde{\Omega}_k$ being} calculated via \eqref{re}-\eqref{coma} backwardly.

 {Moreover, the optimal cost function is given by
 \begin{align}
 J_N^*=&\sum_{i=0}^{L}E[(x_0^i)^TP_0^{i}x_0^i]+\sum_{i=0}^{L}(1-p^i)Tr[\Sigma_{x_0^i}(P_0^{i}+H_0^{i})]\notag\\
  &+\sum_{i=0}^{L}\sum_{k=0}^{N}Tr(\Sigma_{v^i}L^{i}_{k+1}).
 \end{align}}

\end{theorem}
\begin{proof}
  See Appendix.
\end{proof}

\begin{remark}
In Theorem \ref{th2}, we first derive the optimal control strategies by decoupling the G-FBSDEs \eqref{gfbs} (equivalently G-FBSDEs \eqref{gfbs2}). {The optimal control strategies are given in terms of new CREs \eqref{re}, which can be calculated backwardly under Assumptions \ref{ass1}-\ref{ass2} and the conditions {that} $\Lambda_k, \tilde{\Lambda}_k$ are invertible.} Moreover, it is noted that $P_k$ and $H_k$ in \eqref{re} are asymmetric, {which is} the essential difference from the additive noise case, see \cite{aon2018,am2019,lx2018,oan2016,nmt2013}.
\end{remark}

\section{Discussions}
In this section, {we shall discuss some special cases of Problem \ref{prob1} and demonstrate the novelty and significance of our results.}

\subsection{Additive Noise Case}
In the case of $w_k^i=0$, the MN-NCS \eqref{sm-1} turns into the additive noise case. Using the results in Theorem \ref{th2}, the solution to Problem \ref{prob1} can be presented as follows.
\begin{corollary}\label{lem7}
{Under Assumptions \ref{ass1} and \ref{ass2}}, Problem \ref{prob1} is uniquely solvable. Moreover, the control strategies $u^i_k, k=0,\cdots,N, i=0,\cdots,L$ {that minimize \eqref{pi} can be given by}
\begin{equation}\label{uk2}
\left\{ \begin{array}{ll}
  u_k^0&=\mathcal{I}^0\hat{U}_k,\\
 u_k^i&=\mathcal{I}^i\hat{U}_k+\tilde{u}_k^i, i=1,\cdots,L,
\end{array} \right.
\end{equation}
where
  \begin{align}
  \hat{U}_k& =-\Lambda^{-1}_k\Psi_k\hat{X}_k, \label{uk02}\\
   \tilde{u}_k^i&=-(\tilde{\Pi}^i_k)^{-1}\tilde{\Omega}^i_k\tilde{x}_k^i,~~i=1,\cdots,L.\label{uki2}
  \end{align}
 {In the above,} $\hat{X}_k, \tilde{x}_k^i$ are given in Lemma \ref{lemma2} with $w_k^i=0$, and the coefficient matrices $\Lambda_k, \Psi_k, \tilde{\Pi}_k, \tilde{\Omega}_k$ can be calculated via the following Riccati equations:
   {\small\begin{equation}\label{re2}
\left\{ \begin{array}{ll}
  P_k&=Q+A^TP_{k+1}A-\Psi_k^T\Lambda_k^{-1}\Psi_k,\\
    P_k^{i}&=Q^{ii}+(A^i)^TP_{k+1}^{i}A^i
    -(\Omega_k^i)^T(\Pi_k^i)^{-1}\Omega_k^i,\\
   H_k^{i}&=Q^{ii}+(A^i)^TL_{k+1}^{i}A^i
   -(\tilde{\Omega}_k^i)^T(\tilde{\Pi}_k^i)^{-1}\tilde{\Omega}_k^i,\\
 L_{N+1}^{i}&=P_{N+1}^{i}, P_{N+1} ~~\text{given in \eqref{pi}},
\end{array} \right.
\end{equation}}
where $\Lambda_k, \Psi_k, \Pi_k^i, \Omega_k^i, \tilde{\Pi}_k^i, \tilde{\Omega}_k^i, L_k, L_k^{i}$ in \eqref{re} satisfy
{\small\begin{equation}\label{coma2}
\left\{ \begin{array}{ll}
 \Lambda_k&=R+B^TP_{k+1}B,\\
    \Psi_k&=B^TP_{k+1}A,\\
    \Pi^i_k&=R^{ii}+(B^i)^TP_{k+1}^{i}B^i,\\
  \Omega_k^i&=(B^i)^TP_{k+1}^{i}A^i,\\
    \tilde{\Pi}^i_k&=R^{ii}+(B^i)^TL_{k+1}^{i}B^i,\\
  \tilde{\Omega}_k^i&=(B^i)^TL_{k+1}^{i}A^i,\\
  L_k^{i}&=p^iP_k^{i}+(1-p^i)H_k^{i}.
\end{array} \right.
\end{equation}}
\end{corollary}
\begin{remark}

  As shown in Corollary \ref{lem7}, the following points should be noted:
  \begin{itemize}
    \item The obtained results in {Theorem \ref{th2}} can be reduced to the additive noise case (i.e., $w_k^i=0$), which include the results of \cite{aon2018,am2019,lx2018} as a special case.
    \item For the additive noise case, it is found that $L_k=P_k=H_k$, hence CREs \eqref{re} can be reduced to \eqref{re2}, which are symmetric.
    \item For Riccati equations \eqref{re2}, by following the techniques of Lemma \ref{lem6}, it can be shown that $\Lambda_k, \Pi_k^i$ are positive definite under Assumptions \ref{ass1} and \ref{ass2}, thus the solvability of Problem \ref{prob1} can be ensured. Furthermore, the control strategies \eqref{uk2} are unique in this case.
  \end{itemize}
\end{remark}

\subsection{Single Subsystem Case}
In this section, we will consider the single subsystem case, i.e., $L=1$.

Following the results in Theorem \ref{th2}, {the solution to Problem \ref{prob1} for the single subsystem case can be given as follows.}
\begin{corollary}\label{lem8}
Under Assumptions \ref{ass1} and \ref{ass2}, Problem \ref{prob1} is uniquely solvable, {and the optimal controls $u^i_k, k=0,\cdots,N, i=0, 1$ can be given as}
\begin{equation}\label{uk3}
\left\{ \begin{array}{ll}
  u_k^0&=\mathcal{I}^0\hat{U}_k,\\
 u_k^1&=\mathcal{I}^1\hat{U}_k+\tilde{u}_k^1,
\end{array} \right.
\end{equation}
where
  \begin{align}
  \hat{U}_k& =-\Lambda^{-1}_k\Psi_k\hat{X}_k, \label{uk03}\\
   \tilde{u}_k^1&=-(\tilde{\Lambda}_k)^{-1}
   \tilde{\Psi}_k\tilde{x}_k,\label{uki3}
  \end{align}
   {with $\Lambda_k, \Psi_k, \tilde{\Lambda}_k, \tilde{\Psi}_k$ given by}
   {\small\begin{equation}\label{coma3}
\left\{ \begin{array}{ll}
 \Lambda_k&=R+B^TP_{k+1}B+\Sigma_{w^1}\bar{B}^TL_{k+1}\bar{B},\\
    \Psi_k&=B^TP_{k+1}A+\Sigma_{w^1}\bar{B}^TL_{k+1}\bar{A},\\
    \tilde{\Lambda}_k&=R+B^TL_{k+1}B+\Sigma_{w^1}\bar{B}^TL_{k+1}\bar{B},\\
    \tilde{\Psi}_k&=B^TL_{k+1}A+\Sigma_{w^1}\bar{B}^TL_{k+1}\bar{A},\\
  L_k&=p P_k +(1-p)H_k,
\end{array} \right.
\end{equation}}
   {and $P_k, H_k$ satisfying}
   {\small\begin{equation}\label{re3}
\left\{ \begin{array}{ll}
  P_k&\hspace{-1mm}=\hspace{-1mm}Q+A^TP_{k+1}A+\Sigma_{w^1}\bar{A}^TL_{k+1}\bar{A}-\Psi_k^T\Lambda_k^{-1}\Psi_k,\\
   H_k&\hspace{-1mm}=\hspace{-1mm}Q+A^TL_{k+1}A+\Sigma_{w^1}\bar{A}^TL_{k+1}\bar{A}-\tilde{\Psi}_k^T\tilde{\Lambda}_k^{-1}\tilde{\Psi}_k,
\end{array} \right.
\end{equation}}
with terminal conditions $L_{N+1}=H_{N+1}=P_{N+1}$, and $\Sigma_{w^1}$ is the covariance of system noise $\{w_k^1\}_{k=0}^{N}$ with $L=1$ in \eqref{sm-1} and \eqref{pi}.
\end{corollary}

As for the results given in Corollary \ref{lem8}, we have the following comments.
\begin{remark}
  Firstly, different from {the} multiple subsystems case in Theorem \ref{th2}, the optimal controls \eqref{uk3} are based on symmetric Riccati equations \eqref{re3}. Secondly, Assumptions \ref{ass1} and \ref{ass2} are sufficient to guarantee the solvability of Problem \ref{prob1}.
\end{remark}

\subsection{Solvability with Indefinite Weighting Matrices}
In this section, we will investigate the case of {indefinite weighting matrices in \eqref{pi}}. In other words, we will just assume that weighting matrices $Q, R, P_{N+1}$ in \eqref{pi} are {symmetric}.

Firstly, we will introduce the generalized Riccati equation:
{\small\begin{equation}\label{up1k}
\left\{ \begin{array}{ll}
   \Delta_k&=Q+A^T\Delta_{k+1}A+\sum_{i=1}^{L}\Sigma_{w^i}(\bar{\mathbf{A}}^i)^T\Delta_{k+1}
  \bar{\mathbf{A}}^i\\
  &-M_k^T\Upsilon_k^{\dag}M_k,\\
    \Upsilon_k&=R+B^T\Delta_{k+1}B+\sum_{i=1}^{L}\Sigma_{w^i}(\bar{\mathbf{B}}^i)^T\Delta_{k+1}
  \bar{\mathbf{B}}^i,\\
    M_k&=B^T\Delta_{k+1}A+\sum_{i=1}^{L}\Sigma_{w^i}(\bar{\mathbf{B}}^i)^T\Delta_{k+1}
  \bar{\mathbf{A}}^i,
\end{array} \right.
\end{equation}}
{with terminal condition $\Gamma_{N+1}=P_{N+1}$, and $\dag$ denotes the Moore-Penrose inverse}.

We will present the following corollary without proof.
\begin{corollary}\label{lem9}
  Under Assumption \ref{ass1}, Problem \ref{prob1} is uniquely {solvable} if and only if ${\Upsilon}_k\geq 0$ in \eqref{up1k}, and $\Lambda_k, \tilde{\Lambda}_k, \Pi_k^i, \tilde{\Pi}_k^i$ in \eqref{re}-\eqref{coma} are all invertible for $k=0,\cdots, N, L=1,\cdots,L$. In this case, the optimal controls are given by \eqref{uk0}, in which the coefficients are based on the solution to the CREs \eqref{re}.
\end{corollary}

\begin{remark}
 The necessary and sufficient solvability conditions of Problem \ref{prob1} are shown in Corollary \ref{lem9} under the assumption that {the weighting matrices of \eqref{pi} are} indefinite. The proposed results in Corollary \ref{lem9} can be induced from Theorem \ref{th1} and its proof. {It can be observed that the positive semi-definiteness of $\Upsilon_k$ is equivalent to the condition $\delta J_N\geq 0$ in \eqref{deljn}, which is the key of deriving Corollary \ref{lem9}}. {To avoid repetition}, the detailed proof of Corollary \ref{lem9} is omitted here.
\end{remark}

\section{Numerical Examples}
In this section, some numerical simulation examples will be provided to show the effectiveness and feasibilities of the main results.

\subsection{State Trajectory with Optimal Controls}
Consider MN-NCSs \eqref{sm-1} and cost function \eqref{pi} with
{\small\begin{align}\label{coe}
  L&=3, p^1=0.8, p^2=0.6, p^3=0.8,N=60,\notag\\
  A^1&=\left[\hspace{-2mm}
  \begin{array}{cccc}
  2&-2\\
   1&3
  \end{array}
\hspace{-2mm}\right], A^2=\left[\hspace{-2mm}
  \begin{array}{cccc}
  0.8&2\\
   -1&0.6
  \end{array}
\hspace{-2mm}\right], A^3=\left[\hspace{-2mm}
  \begin{array}{cccc}
  1&-0.3\\
   3&-2
  \end{array}
\hspace{-2mm}\right],\notag\\
B^1&=\left[\hspace{-2mm}
  \begin{array}{cccc}
  1.1&-1.9\\
   0.6&2
  \end{array}
\hspace{-2mm}\right],B^2=\left[\hspace{-2mm}
  \begin{array}{cccc}
  1&2\\
   3&4
  \end{array}
\hspace{-2mm}\right], B^3=\left[\hspace{-2mm}
  \begin{array}{cccc}
  -1&1.5\\
   1&2
  \end{array}
\hspace{-2mm}\right],\notag\\
B^{10}&=\left[\hspace{-2mm}
  \begin{array}{cccc}
  -1&1.5\\
   2&-2
  \end{array}
\hspace{-2mm}\right],B^{20}=\left[\hspace{-2mm}
  \begin{array}{cccc}
  0.8&-2\\
   1&-3
  \end{array}
\hspace{-2mm}\right], B^{30}=\left[\hspace{-2mm}
  \begin{array}{cccc}
  1&-3\\
   0&2
  \end{array}
\hspace{-2mm}\right],\notag\\
\bar{A}^1&=\left[\hspace{-2mm}
  \begin{array}{cccc}
  0.6&2\\
   3&0.9
  \end{array}
\hspace{-2mm}\right], \bar{A}^2=\left[\hspace{-2mm}
  \begin{array}{cccc}
  1&-0.5\\
   0.5&2
  \end{array}
\hspace{-2mm}\right], \bar{A}^3=\left[\hspace{-2mm}
  \begin{array}{cccc}
  0&-1\\
   2&2.3
  \end{array}
\hspace{-2mm}\right],\notag\\
\bar{B}^1&=\left[\hspace{-2mm}
  \begin{array}{cccc}
  -1&0.2\\
   3&2
  \end{array}
\hspace{-2mm}\right],\bar{B}^2=\left[\hspace{-2mm}
  \begin{array}{cccc}
  1.2&1.8\\
   -3&2
  \end{array}
\hspace{-2mm}\right], \bar{B}^3=\left[\hspace{-2mm}
  \begin{array}{cccc}
  1.1&-2\\
   3.1&2
  \end{array}
\hspace{-2mm}\right],\notag\\
\bar{B}^{10}&=\left[\hspace{-2mm}
  \begin{array}{cccc}
  -1&0.9\\
   0.7&2
  \end{array}
\hspace{-2mm}\right],\bar{B}^{20}=\left[\hspace{-2mm}
  \begin{array}{cccc}
  1.2&1\\
   -2&0
  \end{array}
\hspace{-2mm}\right], \bar{B}^{30}=\left[\hspace{-2mm}
  \begin{array}{cccc}
  3&-1\\
   0.6&2
  \end{array}
\hspace{-2mm}\right],\notag\\
\Sigma_{w^1}&=\Sigma_{w^2}=\Sigma_{w^3}=1,
\Sigma_{v^1}=\Sigma_{v^2}=\Sigma_{v^3}=\left[\hspace{-2mm}
  \begin{array}{cccc}
  1&0\\
  0&1
  \end{array}
\hspace{-2mm}\right],\notag\\
\mu^1&=\left[\hspace{-2mm}
  \begin{array}{cccc}
  1.2\\
  2
  \end{array}
\hspace{-2mm}\right], \mu^2=\left[\hspace{-2mm}
  \begin{array}{cccc}
  2\\
  5
  \end{array}
\hspace{-2mm}\right], \mu^3=\left[\hspace{-2mm}
  \begin{array}{cccc}
  -3\\
  10
  \end{array}
\hspace{-2mm}\right],\notag\\
\Sigma_{x_0^1}&=\Sigma_{x_0^2}=\Sigma_{x_0^3}=\left[\hspace{-2mm}
  \begin{array}{cccc}
  1&0\\
  0&1
  \end{array}
\hspace{-2mm}\right],\notag\\
Q&=\left[\hspace{-2mm}
  \begin{array}{cccccc}
  1&0&-1&1.2&0.6 &0.2\\
  0&1&2&0.8&1 &1\\
  -1&2&1.2&0.6&1 &1\\
  1.2&0.8&0.6&0.6&3 &1\\
  0.6&1&1&3&2 &1\\
  0.2&1&1&1&1 &2
  \end{array}
\hspace{-2mm}\right],\notag\\
 R&\hspace{-1mm}=\hspace{-1mm}\left[\hspace{-2mm}
  \begin{array}{cccccccc}
  5.7 \hspace{-1mm}&\hspace{-1mm}-1.0 \hspace{-1mm}&\hspace{-1mm}-0.1 \hspace{-1mm}&\hspace{-1mm}-0.4 \hspace{-1mm}&\hspace{-1mm}-3.9 &
  -4.8 \hspace{-1mm}&\hspace{-1mm}-0.9 \hspace{-1mm}&\hspace{-1mm}　1.9 \\
   -1.0\hspace{-1mm}&\hspace{-1mm} 6.5\hspace{-1mm}&\hspace{-1mm}0.5 \hspace{-1mm}&\hspace{-1mm}-4.7 \hspace{-1mm}&\hspace{-1mm}-1.6 &
   -0.4 \hspace{-1mm}&\hspace{-1mm}4.1 \hspace{-1mm}&\hspace{-1mm}0.9　\\
    -0.1\hspace{-1mm}&\hspace{-1mm} 0.5\hspace{-1mm}&\hspace{-1mm} 10.4\hspace{-1mm}&\hspace{-1mm}0.7 &
    2.3 \hspace{-1mm}&\hspace{-1mm}-3.7 \hspace{-1mm}&\hspace{-1mm}-3.9 \hspace{-1mm}&\hspace{-1mm}2.8　\\
     -0.4\hspace{-1mm}&\hspace{-1mm} -4.7\hspace{-1mm}&\hspace{-1mm}0.7 \hspace{-1mm}&\hspace{-1mm}8.9 &
     1.4 \hspace{-1mm}&\hspace{-1mm}4.9 \hspace{-1mm}&\hspace{-1mm}-0.4 \hspace{-1mm}&\hspace{-1mm}3.8　\\
      -3.9\hspace{-1mm}&\hspace{-1mm} -1.6\hspace{-1mm}&\hspace{-1mm}2.3 \hspace{-1mm}&\hspace{-1mm}1.4 &
      9.5 \hspace{-1mm}&\hspace{-1mm}5.9 \hspace{-1mm}&\hspace{-1mm}-2.6 \hspace{-1mm}&\hspace{-1mm}-2.8　\\
       -4.8\hspace{-1mm}&\hspace{-1mm} -0.4\hspace{-1mm}&\hspace{-1mm}-3.7 \hspace{-1mm}&\hspace{-1mm}4.9 \hspace{-1mm}&\hspace{-1mm}5.9 &
       11.3 \hspace{-1mm}&\hspace{-1mm}4.2 \hspace{-1mm}&\hspace{-1mm}-0.7　\\
        -0.9\hspace{-1mm}&\hspace{-1mm} 4.1\hspace{-1mm}&\hspace{-1mm}-3.9 \hspace{-1mm}&\hspace{-1mm}-0.4 \hspace{-1mm}&\hspace{-1mm}-2.6 \hspace{-1mm}&\hspace{-1mm} 4.2\hspace{-1mm}&\hspace{-1mm}13.7 \hspace{-1mm}&\hspace{-1mm}1.6　\\
         1.9\hspace{-1mm}&\hspace{-1mm} 0.9\hspace{-1mm}&\hspace{-1mm} 2.8\hspace{-1mm}&\hspace{-1mm} 3.8\hspace{-1mm}&\hspace{-1mm}-2.8 \hspace{-1mm}&\hspace{-1mm}-0.7 \hspace{-1mm}&\hspace{-1mm}1.6\hspace{-1mm}&\hspace{-1mm}8.3　
         \end{array}
\hspace{-2mm}\right],\notag\\
 P_{101}&=\left[\hspace{-2mm}
  \begin{array}{cccccc}
  1.2&-1.4&-1.4&-3.1&-1.7& -5.5\\
  -1.4&1.3&0.9&1.8&-1.1 &0.7\\
  -1.4&0.9&5.9&1.4&-3.5 &-1.4\\
  -3.1&1.8&1.4&4.4&-0.2 &1.5\\
  -1.7&-1.1&-3.5&-0.2&5.1 &2.7\\
  -5.5&0.7&-1.4 & 1.5&2.7 &10.4
  \end{array}
\hspace{-2mm}\right].
\end{align}}

From Theorem \ref{th2}, since Assumptions \ref{ass1} and \ref{ass2} hold for the coefficients given in \eqref{coe}, and $\Lambda_k$ and $\tilde{\Lambda}_k$ in \eqref{coma} are invertible, hence Problem \ref{prob1} can be uniquely solved. In this case, by using the optimal controls $u_k^i, i=1,2,3, k=0,\cdots,60$, the state trajectories of $x_k^1, x_k^2, x_k^3$ are presented as in Figures 2-4.
 \begin{figure}[htbp]\label{fig2}
  \centering
  \includegraphics[width=0.45\textwidth]{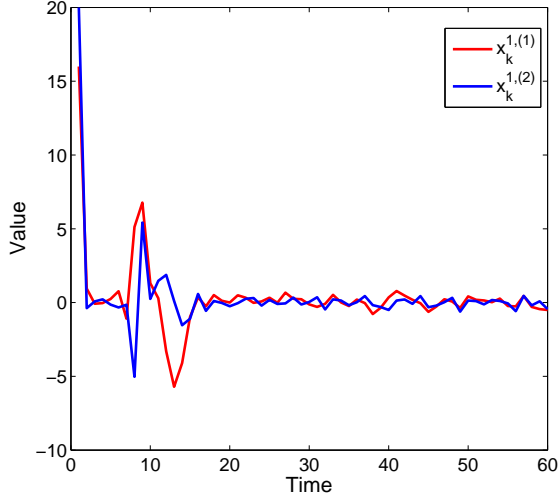}
  \caption{State trajectory of $x_k^1$, where $x_k^{1,(1)}, x_k^{1,(2)}$ are the first and the second element of $x_k^{1}$, respectively.}
\end{figure}
 \begin{figure}[htbp]\label{fig3}
  \centering
  \includegraphics[width=0.45\textwidth]{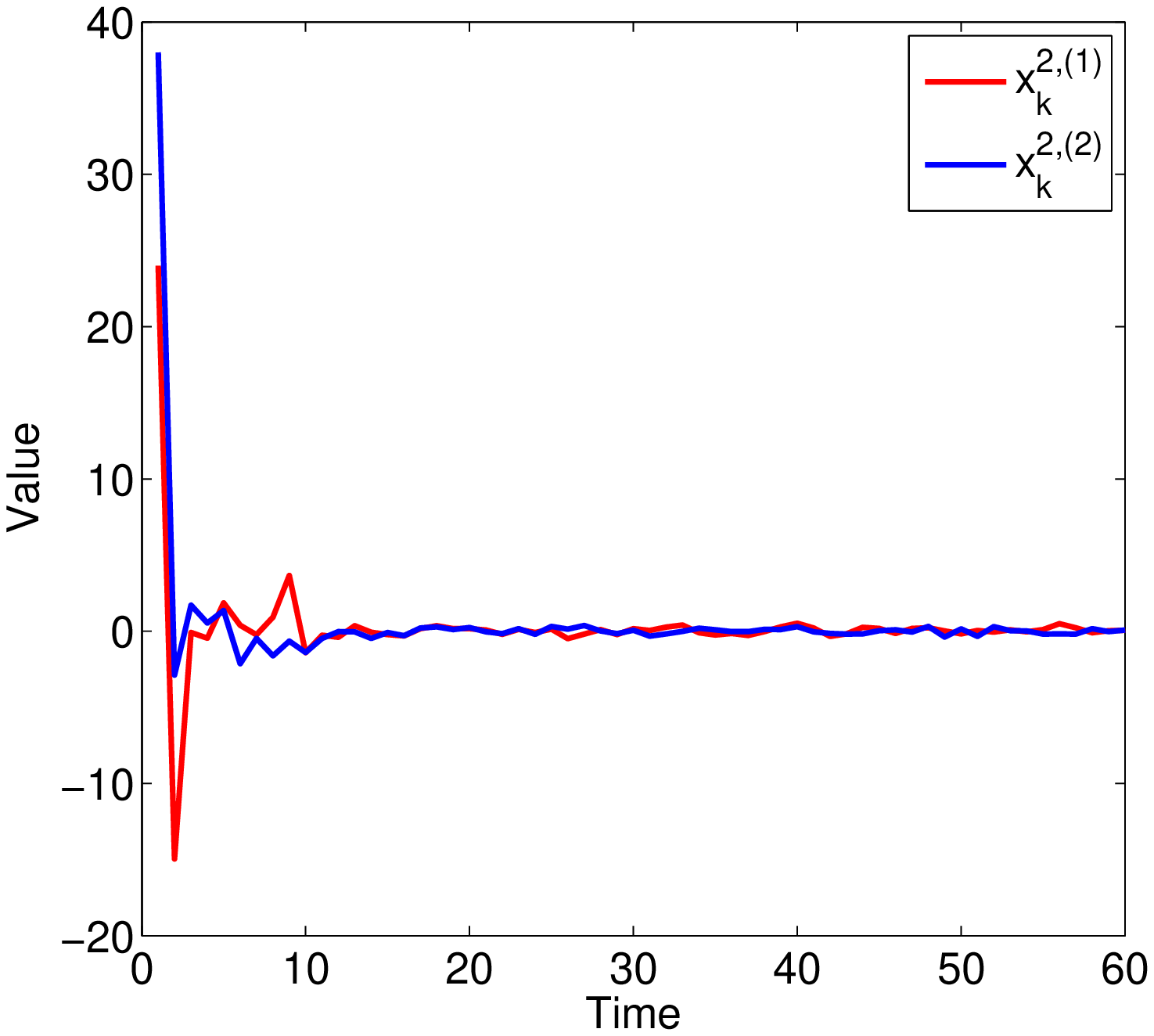}
  \caption{State trajectory of $x_k^2$, $x_k^{2,(1)}, x_k^{2,(2)}$ are the first and the second element of $x_k^{2}$, respectively.}
\end{figure}
 \begin{figure}[htbp]\label{fig4}
  \centering
  \includegraphics[width=0.45\textwidth]{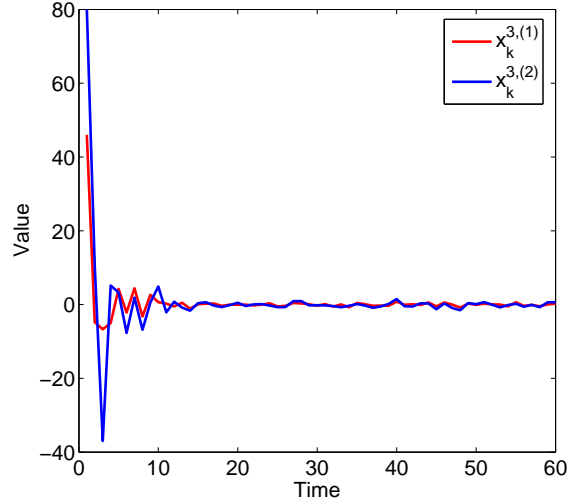}
  \caption{State trajectory of $x_k^3$, $x_k^{3,(1)}, x_k^{3,(2)}$ are the first and the second element of $x_k^{2}$, respectively.}
\end{figure}

As shown, each subsystem state $x_k^i, i=1,2,3$ is convergent with optimal controls $u_k^i, i=0,\cdots,3$ calculated via Theorem \ref{th2}.

\subsection{State Trajectory with Different Packet Dropout Rates}
In this section, we will explore the effects on the state trajectories with different packet dropout rates $p^i, i=1,2,3$.

Without loss of generality, we choose the same coefficients with \eqref{coe}, and the state trajectory of subsystem 1 is given in Figures 5-6 with different packet dropout rates.
 \begin{figure}[htbp]\label{fig5}
  \centering
  \includegraphics[width=0.45\textwidth]{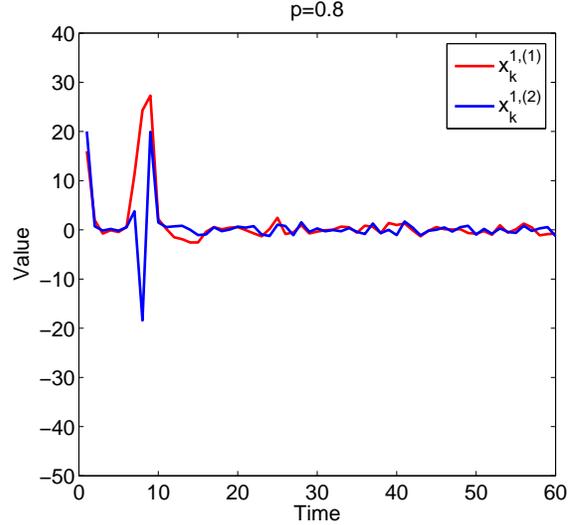}
  \caption{State trajectory of $x_k^1$, with $p^1=p^2=p^3=0.8$.}
\end{figure}
 \begin{figure}[htbp]\label{fig6}
  \centering
  \includegraphics[width=0.45\textwidth]{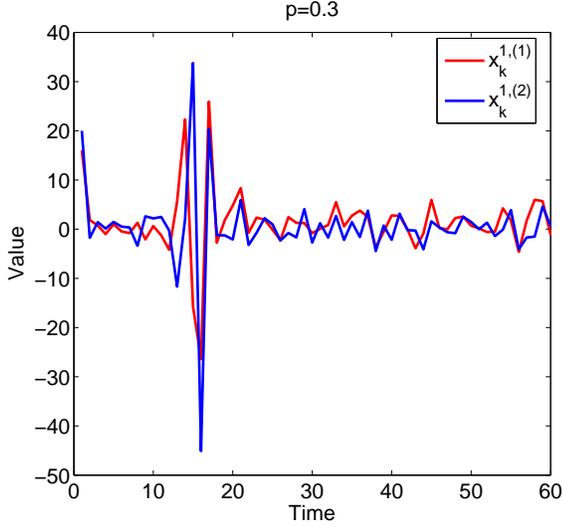}
  \caption{State trajectory of $x_k^1$, with $p^1=p^2=p^3=0.3$.}
\end{figure}

It can be observed that the convergence rate of $x_k^1$ decreases with the packet dropout rate {$1-p^i, i=1,2,3$ becoming larger}.

\section{Conclusion}
In this paper, we have investigated the optimal local and remote control problem for MN-NCSs with unreliable uplink channels and multiple subsystems. By adopting the Pontryagin maximum principle, the necessary and sufficient solvability conditions have been derived. Moreover, we have proposed a novel approach to decouple G-BFSDEs associated with the Pontryagin maximum principle. Finally, by introducing the {asymmetric} CREs, the optimal local and remote control strategies {were} derived in terms of the solution to CREs, which can be calculated backwardly. The proposed methods and the obtained results in this paper provide some inspirations for {studying} general control problems with {asymmetric} information structures.

\section*{Appendix: Proof of Theorem \ref{th2}}

\begin{proof}
  We will show the main results by using an induction method.

  Note that the terminal conditions are $P_{N+1}=H_{N+1}=L_{N+1}$, and $\Theta_N=P_{N+1}X_{N+1}$, it can be calculated from \eqref{adeqs} that
  \begin{align}\label{unm}
   0&=R\hat{U}_N+E[(B+\sum_{i=1}^{L}w_N^i\bar{\mathbf{B}}^i)^T\Theta_N|\mathcal{F}^0_N]\notag\\
   &=R\hat{U}_N+E[(B+\sum_{i=1}^{L}w_N^i\bar{\mathbf{B}}^i)^TP_{N+1}X_{N+1}|\mathcal{F}^0_N]\notag\\
   &=[R+B^TP_{N+1}B+\sum_{i=1}^{L}(\bar{\mathbf{B}}^i)^TP_{N+1}\bar{\mathbf{B}}^i]\hat{U}_N\notag\\
   &+[B^TP_{N+1}A+\sum_{i=1}^{L}(\bar{\mathbf{B}}^i)^TP_{N+1}\bar{\mathbf{A}}^i]E[X_N|\mathcal{F}^0_N]\notag\\
   &=\Lambda_N\hat{U}_N+\Psi_NE[X_N|\mathcal{F}^0_N].
  \end{align}

  From Theorem \ref{th1} we know that Problem \ref{prob1} is uniquely {solvable} if and only if G-FBSDEs \eqref{gfbs2} is uniquely {solvable} under Assumptions \ref{ass1} and \ref{ass2}. Moreover, \eqref{unm} is uniquely solved if and only if $\Lambda_N$ is invertible. In this case, $\hat{U}_N$ can be derived as in \eqref{uk0}.

  Next, from \eqref{adeqs2} we know
  \begin{align}\label{theni}
   \Theta_N^i&=P_{N+1}^{i1}x_{N+1}^1+\cdots+P_{N+1}^{iL}x_{N+1}^L,
  \end{align}
 then there holds
  \begin{align}\label{unm2}
    0&=R^{ii}\tilde{u}_N^i+E[(B^i+w_N^i\bar{B}^i)^T\sum_{j=1}^{L}P_{N+1}^{ij}x_{N+1}^j|\mathcal{G}_N]\notag\\
  &~~~~~~~~~~~~~-E[(B^i+w_N^i\bar{B}^i)^T\sum_{j=1}^{L}P_{N+1}^{ij}x_{N+1}^j|\mathcal{G}_N^i]\notag\\
  &=R^{ii}\tilde{u}_N^i+E[(B^i+w_N^i\bar{B}^i)^T\sum_{j=1}^{L}P_{N+1}^{ij}
  [(A^j+w_k^j\bar{A})x_{N}^j\notag\\
  &~~~+(B^j+w_k^j\bar{B})u_{N}^j+(B^{j0}+w_k^j\bar{B}^{j0})u_N^0]|\mathcal{G}_N]\notag\\
  &-E[(B^i+w_N^i\bar{B}^i)^T\sum_{j=1}^{L}P_{N+1}^{ij} [(A^j+w_k^j\bar{A})x_{N}^j\notag\\
  &~~~+(B^j+w_k^j\bar{B})u_{N}^j+(B^{j0}+w_k^j\bar{B}^{j0})u_N^0]|\mathcal{G}_N^i]\notag\\
  &=R^{ii}\tilde{u}_N^i+E[(B^i+w_N^i\bar{B}^i)^TP_{N+1}^{i}
  [(A^i+w_k^i\bar{A})x_{N}^i\notag\\
  &~~~+(B^i+w_k^i\bar{B})u_{N}^i+(B^{j0}+w_k^i\bar{B}^{j0})u_N^0]|\mathcal{G}_N]\notag\\
  &-E[(B^i+w_N^i\bar{B}^i)^TP_{N+1}^{i} [(A^i+w_k^i\bar{A})x_{N}^i\notag\\
  &~~~+(B^i+w_k^i\bar{B})u_{N}^i+(B^{j0}+w_k^j\bar{B}^{j0})u_N^0]|\mathcal{G}_N^i]\notag\\
  &=[R^{ii}+(B^i)^TP_{N+1}^{i}B^i+\Sigma_{w^i}(\bar{B}^i)^TP_{N+1}^{i}\bar{B}^i]\tilde{u}_N^i\notag\\
  &+[(B^i)^TP_{N+1}^{i}A^i+\Sigma_{w^i}(\bar{B}^i)^TP_{N+1}^{i}\bar{A}^i]\tilde{x}_{N}^i\notag\\
  &=\tilde{\Pi}_N^i\tilde{u}_N^i+\tilde{\Omega}_N^i\tilde{x}_N^i.
  \end{align}
  Since $\tilde{\Pi}_N^i$ is positive definite as shown in Lemma \ref{lem6}, thus $\tilde{u}_N^i$ can be uniquely solved as \eqref{uki}. {Hence, the optimal remote controller $u_N^0$ and the optimal local controllers $u_N^i, i=1,\cdots, L$ can be derived as \eqref{uk}.
}

  Consequently, we will calculate $\Theta_{N-1}$, actually, from \eqref{coss} we have
  \begin{align}\label{then1}
    \Theta_{N-1}&=E\left[(A+\sum_{i=1}^{L}w_N^i\bar{\mathbf{A}}^i)^T\Theta_N+QX_N\Bigg|\mathcal{G}_N\right]\notag\\
    &=E\left[(A+\sum_{i=1}^{L}w_N^i\bar{\mathbf{A}}^i)^TP_{N+1}X_{N+1}+QX_N\Bigg|\mathcal{G}_N\right]\notag\\
    &=E\Big[(A+\sum_{i=1}^{L}w_N^i\bar{\mathbf{A}}^i)^TP_{N+1}[(A+\sum_{i=1}^{L}w_N^i\bar{\mathbf{A}}^i)X_N\notag\\
    &+(B+\sum_{i=1}^{L}w_N^i\bar{\mathbf{B}}^i)(\hat{U}_N+\tilde{U}_N)+V_N]\Big|\mathcal{G}_N\Big]+QX_N\notag\\
    &=[Q+A^TP_{N+1}A+\sum_{i=1}^{L}\Sigma_{w_N^i}(\bar{\mathbf{A}}^i)^TP_{N+1}\bar{\mathbf{A}}^i]X_N\notag\\
    &+[A^TP_{N+1}B+\sum_{i=1}^{L}\Sigma_{w_N^i}(\bar{\mathbf{A}}^i)^TP_{N+1}\bar{\mathbf{B}}^i]\hat{U}_N\notag\\
    &+[A^TP_{N+1}B+\sum_{i=1}^{L}\Sigma_{w_N^i}(\bar{\mathbf{A}}^i)^TP_{N+1}\bar{\mathbf{B}}^i]\tilde{U}_N\notag\\
    &=[Q+A^TP_{N+1}A+\sum_{i=1}^{L}\Sigma_{w_N^i}(\bar{\mathbf{A}}^i)^TP_{N+1}\bar{\mathbf{A}}^i]X_N\notag\\
    &-\Psi_N^T\Lambda_N^{-1}\Psi_N\hat{X}_N-\Psi_N^T\Lambda_N^{-1}\Psi_N\tilde{X}_N\notag\\
    &=P_N\hat{X}_N+H_N\tilde{X}_N.
  \end{align}
  where $P_N, H_N$ satisfy \eqref{re} for $k=N$.

To complete the induction approach, we assume for $k=l+1,\cdots, N$, there holds
\begin{itemize}
  \item The optimal controls $u_k^i, i=0,\cdots, L$ are given by \eqref{uk}-\eqref{uki};
  \item The relationship between the system state $X_k$ and costate $\Theta_k$ satisfies:
      \begin{align}\label{rela12}
        \Theta_{k-1}=P_k\hat{X}_{k}+H_k\tilde{X}_k,
      \end{align}
      where $P_k, H_k$ can be calculated from \eqref{re} backwardly.
\end{itemize}

  Next, we will calculate $u^i_{l}$.
   Note that $\Theta_{l}=P_{l+1}\hat{X}_{l+1}+H_{l}\tilde{X}_{l+1}$, it can be calculated from \eqref{adeqs} that
  \begin{align}\label{1unm}
   0&=R\hat{U}_{l}+E[(B+\sum_{i=1}^{L}w_l^i\bar{\mathbf{B}}^i)^T\Theta_{l}|\mathcal{F}^0_{l}]\notag\\
   &=R\hat{U}_{l}+E[(B+\sum_{i=1}^{L}w_l^i\bar{\mathbf{B}}^i)^T(P_{l+1}\hat{X}_{l+1}\notag\\
   &~~~~~~~~~~+H_{l+1}\tilde{X}_{l+1})|\mathcal{F}^0_{l}]\notag\\
   &=R\hat{U}_{l}+E\Big[(B+\sum_{i=1}^{L}w_l^i\bar{\mathbf{B}}^i)^T\big\{P_{l+1}[\Gamma_{l+1}X_{l+1}\notag\\
   &~~~~~+(I_{\mathbb{N}_L}-\Gamma_{l+1})(A\hat{X}_{l}+B\hat{U}_{l})]\big\}\notag\\
   &+(B+\sum_{i=1}^{L}w_l^i\bar{\mathbf{B}}^i)^T\big\{H_{l+1}(I_{\mathbb{N}_L}-\Gamma_{l+1})[A\tilde{X}_{l}+B\tilde{U}_{l}\notag\\
   &~~~~~+V_{l}
   +\sum_{i=1}^{L}w_l^i
   (\bar{\mathbf{A}}^iX_{l}+\bar{\mathbf{B}}^iU_{l})]\big\}|\mathcal{F}^0_{l}\Big]\notag\\
   &=\{R+B^TP_{l+1}[p B+(I_{\mathbb{N}_L}-p) B]\notag\\
   &~~~~~~+\sum_{i=1}^{L}\Sigma_{w_l^i}(\bar{\mathbf{B}}^i)^T[P_{l+1}p+H_{l+1}(I_{\mathbb{N}_L}-p)]\bar{\mathbf{B}}^i\}\hat{U}_{l}\notag\\
   &+\{B^TP_{l+1}[p A+(I_{\mathbb{N}_L}-p) A]\notag\\
   &~~~~~~+\sum_{i=1}^{L}\Sigma_{w_l^i}(\bar{\mathbf{B}}^i)^T[P_{l+1}p+H_{l+1}(I_{\mathbb{N}_L}-p)]\bar{\mathbf{B}}^i\}\hat{X}_{l}\notag\\
   &=\Lambda_l\hat{U}_l+\Psi_l\hat{X}_l.
   \end{align}
  By following the discussions below \eqref{unm}, we know that \eqref{1unm} can be uniquely solved if and only if $\Lambda_l$ is invertible. {Then} $\hat{U}_l$ can be derived as in \eqref{uk0}.

  Since $\tilde{U}_l=U_l-\hat{U}_l$, we have
  \begin{align}\label{ads1}
    0 & =R\tilde{U}_{l}+E[(B+\sum_{i=1}^{L}w_l^i\bar{\mathbf{B}}^i)^T\Theta_{l}|\mathcal{G}_{l}]\notag\\
    &~~~~~~~~~-E[(B+\sum_{i=1}^{L}w_l^i\bar{\mathbf{B}}^i)^T\Theta_{l}|\mathcal{F}^0_{l}]\notag\\
    &=R\tilde{U}_{l}+E\Big[(B+\sum_{i=1}^{L}w_l^i\bar{\mathbf{B}}^i)^T\big\{P_{{l+1}}[\Gamma_{{l+1}}X_{{l+1}}\notag\\
   &+(I_{\mathbb{N}_L}-\Gamma_{{l+1}})(A\hat{X}_{l}+B\hat{U}_{l})]\big\}\notag\\
   &+(B+\sum_{i=1}^{L}w_l^i\bar{\mathbf{B}}^i)^T\big\{H_{{l+1}}(I_{\mathbb{N}_L}-\Gamma_{{l+1}})[A\tilde{X}_{l}+B\tilde{U}_{l}\notag\\
   &+V_{l}+\sum_{i=1}^{L}w_l^i(\bar{\mathbf{A}}^iX_{l}+\bar{\mathbf{B}}^iU_{l})]\big\}|\mathcal{G}_{l}\Big]\notag\\
   &-E\Big[(B+\sum_{i=1}^{L}w_l^i\bar{\mathbf{B}}^i)^T\big\{P_{{l+1}}[\Gamma_{{l+1}}X_{{l+1}}\notag\\
   &+(I_{\mathbb{N}_L}-\Gamma_{{l+1}})(A\hat{X}_{l}+B\hat{U}_{l})]\big\}\notag\\
   &+(B+\sum_{i=1}^{L}w_l^i\bar{\mathbf{B}}^i)^T\big\{H_{{l+1}}(I_{\mathbb{N}_L}-\Gamma_{{l+1}})[A\tilde{X}_{l}+B\tilde{U}_{l}\notag\\
   &+V_{l}+\sum_{i=1}^{L}w_l^i(\bar{\mathbf{A}}^iX_{l}+\bar{\mathbf{B}}^iU_{l})]\big\}|\mathcal{F}^0_{l}\Big]\notag\\
    &=R\tilde{U}_{l}+[B^TP_{l+1}p A+\sum_{i=1}^{L}\Sigma_{w_l^i}(\bar{\mathbf{B}}^i)^TP_{l+1}p \bar{\mathbf{A}}^i]X_{l}\notag\\
    &+[B^TP_{l+1}p B+\sum_{i=1}^{L}\Sigma_{w_l^i}(\bar{\mathbf{B}}^i)^TP_{l+1}p \bar{\mathbf{B}}^i]U_{l}\notag\\
    &+B^TP_{l+1}(I_{\mathbb{N}_L}-p) A\hat{X}_{l}+B^TP_{l+1}(I_{\mathbb{N}_L}-p) B\hat{U}_{l}\notag\\
    &+B^TH_{l+1}(I_{\mathbb{N}_L}-p) A\tilde{X}_{l}
    +B^TH_{l+1}(I_{\mathbb{N}_L}-p) B\tilde{U}_{l}\notag\\
    &+\sum_{i=1}^{L}\Sigma_{w_l^i}(\bar{\mathbf{B}}^i)^TH_{l+1}(I_{\mathbb{N}_L}-p) \bar{\mathbf{A}}^iX_{l}\notag\\
    &+\sum_{i=1}^{L}\Sigma_{w_l^i}(\bar{\mathbf{B}}^i)^TH_{l+1}(I_{\mathbb{N}_L}-p) \bar{\mathbf{B}}^iU_{l}\notag\\
    &-\{B^TP_{l+1}(p B+(I_{\mathbb{N}_L}-p) B)\notag\\
   & +\sum_{i=1}^{L}\Sigma_{w_l^i}(\bar{\mathbf{B}}^i)^T[P_{l+1}p +H_{l+1}(I_{\mathbb{N}_L}-p)]\bar{\mathbf{B}}^i\}\hat{U}_{l}\notag\\
   &-\{B^TP_{l+1}(p A+(I_{\mathbb{N}_L}-p) A)\notag\\
   &+\sum_{i=1}^{L}\Sigma_{w_l^i}(\bar{\mathbf{B}}^i)^T[P_{l+1}p+H_{l+1}(I_{\mathbb{N}_L}-p)]\bar{A}\}\hat{X}_{l}\notag\\
   &=\{R+B^T[P_{l+1}p+H_{l+1}(I_{\mathbb{N}_L}-p) ]B\notag\\
   &+\sum_{i=1}^{L}\Sigma_{w_l^i}(\bar{\mathbf{B}}^i)^T[P_{l+1}p +H_{l+1}(I_{\mathbb{N}_L}-p) ]\bar{\mathbf{B}}^i\}\tilde{U}_{l}\notag\\
   &+\{B^T[P_{l+1}p+H_{l+1}(I_{\mathbb{N}_L}-p) ]A\notag\\
   &+\sum_{i=1}^{L}\Sigma_{w_l^i}(\bar{\mathbf{B}}^i)^T[P_{l+1}p+H_{l+1}(I_{\mathbb{N}_L}-p)]\bar{\mathbf{A}}^i\}\tilde{X}_{l}\notag\\
   &=\tilde{\Lambda}_{l}\tilde{U}_{l}+\tilde{\Psi}_{l}\tilde{X}_{l}.
  \end{align}
Hence, the solvability of \eqref{ads1} is equivalent to the invertibility of $\tilde{\Lambda}_{l}$, and we have
  \begin{align}\label{tildeu}
   \tilde{U}_{l}=-\tilde{\Lambda}_{l}^{-1}\tilde{\Psi}_{l}\tilde{X}_{l},
  \end{align}
  i.e, \eqref{uki} can be verified for $k=l$.

  Next, from \eqref{adeqs2} we know
  \begin{align}\label{1theni}
   \Theta_{l}^i&=P_{{l+1}}^{i1}\hat{x}_{{l+1}}^1+\cdots+P_{{l+1}}^{iL}\hat{x}_{{l+1}}^L\notag\\
   &+H_{{l+1}}^{i1}\tilde{x}_{{l+1}}^1+\cdots+H_{{l+1}}^{iL}\tilde{x}_{{l+1}}^L.
  \end{align}
 {Thus, using Lemmas \ref{lem3}-\ref{lem4}, there holds}
  \begin{align}\label{1unm2}
    0&=R^{ii}\tilde{u}_{l}^i+E[(B^i+w_{l}^i\bar{B}^i)^T\Theta_{l}^i|\mathcal{G}_{l}]\notag\\
  &~~~~~~~~~~~~~-E[(B^i+w_{l}^i\bar{B}^i)^T\Theta_{l}^i|\mathcal{G}_{l}^i]\notag\\
  &=R^{ii}\tilde{u}_{l}^i+E\Big[(B^i+w_{l}^i\bar{B}^i)^TP_{{l+1}}^{i}
  \{\gamma_{{l+1}}^i[(A^i+w_{l}^i\bar{A})x_{l}^i\notag\\
  &~~~+(B^i+w_{l}^i\bar{B})u_{l}^i+(B^{i0}+w_{l}^i\bar{B}^{i0})u_{l}^0]\notag\\
  &+(1-\gamma_{l+1}^i)(A^i\hat{x}_{l}^i+B^i\hat{u}_{l}^i+B^{i0}u_{l}^0)\}\Big|\mathcal{G}_{l}\Big]\notag\\
  &+E\Big[(B^i+w_{l}^i\bar{B}^i)^TH_{{l+1}}^{i}
  (1-\gamma_{{l+1}}^i)[A^i\tilde{x}_{l}^i+B^i\tilde{u}_{l}^i\notag\\
  &+w_{l}^i(\bar{A}x_{{l+1}}^i+\bar{B}u_{l}^i+\bar{B}^{j0})u_{l+1}^0]\Big|\mathcal{G}_{l}\Big]\notag\\
  &-E\Big[(B^i+w_{l}^i\bar{B}^i)^TP_{{l+1}}^{i}
  \{\gamma_{{l+1}}^i[(A^i+w_k^i\bar{A})x_{{l+1}}^i\notag\\
  &~~~+(B^i+w_k^i\bar{B})u_{{l+1}}^i+(B^{j0}+w_k^i\bar{B}^{j0})u_{l+1}^0]\notag\\
  &+(1-\gamma_{l+1}^i)(A^i\hat{x}_{l}^i+B^i\hat{u}_{l}^i+B^{i0}u_{l}^0)\}\Big|\mathcal{G}_{l}^i\Big]\notag\\
  &-E\Big[(B^i+w_{l}^i\bar{B}^i)^TH_{{l+1}}^{i}
  (1-\gamma_{{l+1}}^i)[A^i\tilde{x}_{l}^i+B^i\tilde{u}_{l}^i\notag\\
  &+w_{l}^i(\bar{A}x_{{l+1}}^i+\bar{B}u_{l}^i+\bar{B}^{j0})u_{l+1}^0]\Big|\mathcal{G}_{l}^i\Big]\notag\\
  &=[R^{ii}+(B^i)^T(p^iP_{{l+1}}^{i}+(1-p^i)H_{l+1}^{i})B^i\notag\\
  &+\Sigma_{w^i}(\bar{B}^i)^T(p^iP_{{l+1}}^{i}+(1-p^i)H_{l+1}^{i})\bar{B}^i]\tilde{u}_{l}^i\notag\\
  &+[(B^i)^T(p^iP_{{l+1}}^{i}+(1-p^i)H_{l+1}^{i})A^i\notag\\
  &+\Sigma_{w^i}(\bar{B}^i)^T(p^iP_{{l+1}}^{i}+(1-p^i)H_{l+1}^{i})\bar{A}^i]\tilde{x}_{l}^i\notag\\
  &=\Pi_{l}^i\tilde{u}_{l}^i+\Omega_{l}^i\tilde{x}_{l}^i.
  \end{align}
 In this case, $\tilde{u}_l^i$ can be derived as \eqref{uki} for $k=l$. {Therefore, the optimal controls $u_l^i, i=1,\cdots,L$ can be verified as \eqref{uk}.
}

  Consequently, we will calculate $\Theta_{l-1}$, from \eqref{coss}
  \begin{align}\label{2then1}
    \Theta_{l-1}&=E\left[(A+\sum_{i=1}^{L}w_l^i\bar{\mathbf{A}}^i)^T\Theta_l+QX_l\Bigg|\mathcal{G}_l\right]\notag\\
    &=E\Big[(A+\sum_{i=1}^{L}w_l^i\bar{\mathbf{A}}^i)^T(P_{l+1}\hat{X}_{l+1}
    \hspace{-1mm}+\hspace{-1mm}H_{l+1}\tilde{X}_{l+1})\notag\\
    &+QX_l|\mathcal{G}_l\Big]\notag\\
    &=QX_l+E\Big[(A+\sum_{i=1}^{L}w_l^i\bar{\mathbf{A}}^i)^TP_{l+1}\notag\\&\times\big[\Gamma_{l+1}((A+\sum_{i=1}^{L}w_l^i\bar{\mathbf{A}}^i)X_{l}\notag\\
    &+(B+\sum_{i=1}^{L}w_l^i\bar{\mathbf{B}}^i)(\hat{U}_l+\tilde{U}_l)+V_l)\notag\\
   &+(I_{\mathbb{N}_l}-\Gamma_{l+1})(A\hat{X}_{l}+B\hat{U}_{l})\big]\notag\\
   &+(A+\sum_{i=1}^{L}w_l^i\bar{\mathbf{A}}^i)^TH_{l+1}(I_{\mathbb{N}_l}-\Gamma_{l+1})[A\tilde{X}_{l}+B\tilde{U}_{l}\notag\\
   &+V_{l}+\sum_{i=1}^{L}w_l^i(\bar{\mathbf{A}}^iX_{l}+\bar{\mathbf{B}}^i(\hat{U}_{l}+\tilde{U}_l))]\Big|\mathcal{G}_l\Big]\notag\\
   &=[Q+A^TP_{l+1}pA
   +\sum_{i=1}^{L}\Sigma_{w_l^i}(\bar{\mathbf{A}}^i)^TP_{l+1}p\bar{\mathbf{A}}^i]X_l\notag\\
   &+(A^TP_{l+1}p B+\sum_{i=1}^{L}\Sigma_{w_l^i}(\bar{\mathbf{A}}^i)^TP_{l+1}p\bar{\mathbf{B}}^i)\hat{U}_l\notag\\
   &+(A^TP_{l+1}p B+\sum_{i=1}^{L}\Sigma_{w_l^i}(\bar{\mathbf{A}}^i)^TP_{l+1}p\bar{\mathbf{B}}^i)\tilde{U}_l\notag\\
   &+A^TP_{l+1}(I_{\mathbb{N}_L}-p)A\hat{X}_l+A^TP_{l+1}(I_{\mathbb{N}_L}-p)B\hat{U}_l\notag\\
   &+A^TH_{l+1}(I_{\mathbb{N}_L}-p)A\tilde{X}_l+A^TH_{l+1}(I_{\mathbb{N}_L}-p)B\tilde{U}_l\notag\\
   &+\sum_{i=1}^{L}\Sigma_{w_l^i}(\bar{\mathbf{A}}^i)^TH_{l+1}(I_{\mathbb{N}_L}-p)\bar{\mathbf{A}}^iX_l
   \notag\\&+\sum_{i=1}^{L}\Sigma_{w_l^i}(\bar{\mathbf{A}}^i)^TH_{l+1}(I_{\mathbb{N}_L}-p)\bar{\mathbf{B}}^i\hat{U}_l\notag\\
   &+\sum_{i=1}^{L}\Sigma_{w_l^i}(\bar{\mathbf{A}}^i)^TH_{l+1}(I_{\mathbb{N}_L}-p)\bar{\mathbf{B}}^i\tilde{U}_l\notag\\
   &=[Q+A^TP_{l+1}pA+\sum_{i=1}^{L}\Sigma_{w_l^i}(\bar{\mathbf{A}}^i)^TP_{l+1}p\bar{\mathbf{A}}^i\notag\\
   &+A^TP_{l+1}(I_{\mathbb{N}_L}-p)A\notag\\&+\sum_{i=1}^{L}\Sigma_{w_l^i}(\bar{\mathbf{A}}^i)^TH_{l+1}(I_{\mathbb{N}_L}-p)\bar{\mathbf{A}}^i]\hat{X}_l\notag\\
   &+\Psi_l\hat{U}_l
   +[Q+A^TP_{l+1}p A+\sum_{i=1}^{L}\Sigma_{w_l^i}(\bar{\mathbf{A}}^i)^TP_{l+1}p\bar{\mathbf{A}}^i\notag\\
   &+A^TH_{l+1}(I_{\mathbb{N}_L}-p)A\notag\\&+\sum_{i=1}^{L}\Sigma_{w_l^i}(\bar{\mathbf{A}}^i)^TH_{l+1}(I_{\mathbb{N}_L}-p)\bar{\mathbf{A}}^i]\tilde{X}_l
   \hspace{-1mm}+\hspace{-1mm}\tilde{\Psi}_l\tilde{U}_l\notag\\
    &=P_l\hat{X}_l+H_l\tilde{X}_l.
  \end{align}
Thus, \eqref{rela12} has been derived for $k=l$. {This ends the induction}.

Finally, we will calculate the optimal cost function. For simplicity, we denote
\begin{align}\label{vn}
  V_{k}&\triangleq E[X_k^T\Theta_{k-1}].
\end{align}

Hence, we have
\begin{align}\label{mivk}
  &V_k-V_{k+1}=E[X_k^T\Theta_{k-1}]-E[X_{k+1}^T\Theta_{k}]\notag\\
  =&E\Big[E[X_k^T(A+\sum_{i=1}^{L}w_k^i\bar{\mathbf{A}}^i)^T\Theta_k+QX_k|\mathcal{G}_k]\notag\\
  &-E\big[[(A+\sum_{i=1}^{L}w_k^i\bar{\mathbf{A}}^i)X_k\notag\\
  &~~~+(B+\sum_{i=1}^{L}w_k^i\bar{\mathbf{B}}^i)U_k+V_k]^T\Theta_{k}|\mathcal{G}_k\big]\Big]\notag\\
  =&E\Big[X_k^TQX_k-V_k^T\Theta_k-U_k^T E[(B+\sum_{i=1}^{L}w_k^i\bar{\mathbf{B}}^i)^T\Theta_k|\mathcal{G}_k]\Big]\notag\\
  =&E\Big[X_k^TQX_k+U_k^TRU_k-V_k^T\Theta_k\Big].
\end{align}

Note $\Theta_N=P_{N+1}X_{N+1}$, then taking summation of \eqref{mivk} from $0$ to $N$, the optimal cost function can be given by
\begin{align}\label{ocp}
  J_N^*=&\sum_{k=0}^{N}E[V_k^T\Theta_k]+E[X_0^T\Theta_{-1}]\notag\\
  =&\sum_{k=0}^{N}E[V_k^T(P_{k+1}\hat{X}_{k+1}+H_{k+1}\tilde{X}_{k+1})]\notag\\
  &+E[X_0^T(P_0\hat{X}_0+H_0\tilde{X}_k)]\notag\\
  =&\sum_{k=0}^{N}E[V_k^T(P_{k+1}p+H_{k+1}(I_{\mathbb{N}_L}-p))V_k]\notag\\
  &+E[\hat{X}_0^TP_0\hat{X}_0+\tilde{X}_0^TH_0\tilde{X}_0]\notag\\
  =&\sum_{k=0}^{N}E[V_k^TL_{k+1}V_k]+E[\hat{X}_0^TP_0\hat{X}_0
  +\sum_{i=0}^{L}(\tilde{x}_0^i)^TH_0^{i}\tilde{x}_0^i]\notag\\
  =&\sum_{k=0}^{N}\sum_{i=0}^{L}E[(v_k^i)^TL^{i}_{k+1}v_k^i]\notag\\
  &+E[\hat{X}_0^TP_0\hat{X}_0+\sum_{i=0}^{L}(\tilde{x}_0^i)^TH_0^{i}\tilde{x}_0^i]\notag\\
  =&\sum_{i=0}^{L}E[(x_0^i)^TP_0^{i}x_0^i]+\sum_{i=0}^{L}(1-p^i)Tr[\Sigma_{x_0^i}(P_0^{i}+H_0^{i})]\notag\\
  &+\sum_{i=0}^{L}\sum_{k=0}^{N}Tr(\Sigma_{v^i}L^{i}_{k+1}).
\end{align}

The proof is complete.
\end{proof}

\end{document}